\documentclass[Journal,10pt]{IEEEtran}

\usepackage{amsmath, amssymb, epsfig, graphicx, bm}
\usepackage{cases}
\usepackage{float, subfigure}
\usepackage{color}
\usepackage{amsthm}
\usepackage{amsfonts,dsfont}
\usepackage{cite, url}
\usepackage[all,cmtip]{xy}
\usepackage{color,hyperref}
\usepackage{algorithmic,algorithm}

\definecolor{darkblue}{rgb}{0,0,1}
\hypersetup{colorlinks,breaklinks,
linkcolor=darkblue,urlcolor=darkblue,anchorcolor=darkblue,citecolor=darkblue}

\usepackage[normalem]{ulem} 

\newtheorem{theorem}{Theorem}

\newtheorem{lemma}{Lemma}
\newtheorem{remark}{Remark}

\newtheorem{assumption}{Assumption}

\newcommand{\ba}{\left[ \begin{array}}
\newcommand{\ea}{\\ \end{array} \right]}

\newcommand{\DE}{{\Delta}}

\newcommand{\A}{{\alpha}}
\newcommand{\N}{{\nabla}}

\newcommand{\F}{\frac}

\newcommand{\LA}{\langle}
\newcommand{\RA}{\rangle}
\newcommand{\TI}{\tilde}
\newcommand{\R}{{\mathbb{R}}}

\newcommand{\vx}{{\mathbf{x}}}
\newcommand{\vy}{{\mathbf{y}}}

\newcommand{\vW}{{\mathbf{W}}}
\newcommand{\vR}{{\mathbf{R}}}

\allowdisplaybreaks

\begin{document}

\title{Can Primal Methods Outperform Primal-dual Methods in Decentralized Dynamic Optimization?}
\author{\authorblockN{Kun Yuan$^\dagger$, Wei Xu$^\ddagger$}, Qing Ling$^*$
\thanks{Kun Yuan is with Department of Electrical and Computer Engineering, University of California, Los Angeles. Wei Xu is with Department of Automation, University of Science and Technology of China. Qing Ling is with School of Data and Computer Science and Guangdong Province Key Laboratory of Computational Science, Sun Yat-Sen University. This work is supported in part by NSF China Grants 61573331 and 61973324, and Fundamental Research Funds for the Central Universities. A preliminary version of this paper has been published in Asilomar Conference on Signals, Systems, and Computers, Pacific Grove, USA, November 3-6, 2019. Corresponding email: lingqing556@mail.sysu.edu.cn. }	
}


\maketitle

\begin{abstract}
In this paper, we consider the decentralized dynamic optimization problem defined over a multi-agent network. Each agent possesses a time-varying local objective function, and all agents aim to collaboratively track the drifting global optimal solution that minimizes the summation of all local objective functions. The decentralized dynamic optimization problem can be solved by primal or primal-dual methods, and when the problem degenerates to be static, it has been proved in literature that primal-dual methods are superior to primal ones. This motivates us to ask: \textit{are primal-dual methods necessarily better than primal ones in decentralized dynamic optimization?}

	
To answer this question, we investigate and compare convergence properties of the primal method, diffusion, and the primal-dual approach, decentralized gradient tracking (DGT). Theoretical analysis reveals that diffusion can outperform DGT in certain dynamic settings. We find that DGT and diffusion are significantly affected by the drifts and the magnitudes of optimal gradients, respectively. In addition, we show that DGT is more sensitive to the network topology, and a badly-connected network can greatly deteriorate its convergence performance. These conclusions provide guidelines on how to choose proper dynamic algorithms in various application scenarios. Numerical experiments are constructed to validate the theoretical analysis.
\end{abstract}

\begin{IEEEkeywords}
Decentralized dynamic optimization, diffusion, decentralized gradient tracking (DGT) 
\end{IEEEkeywords}

\IEEEpeerreviewmaketitle

\section{Introduction}
\label{sec:intro}
Consider a bidirectionally connected network consisting of $n$ agents. At every time $k$, these agents collaboratively solve a decentralized dynamic optimization problem in the form of
\begin{align}\label{eq1}
	\min\limits_{\TI{x}\in \R^d} \quad \sum_{i=1}^n f^k_i(\TI{x}).
\end{align}
Here, $f^k_i:\R^d\rightarrow \R$ is a convex and smooth local objective function which is only available to agent $i$ at time $k$. The optimization variable $\TI{x}\in\R^d$ is common to all agents, and $\TI{x}^{k*}\in\R^d$ is the optimal solution to \eqref{eq1} at time $k$. Our goal is to find  $\TI{x}^{k*}$ at every time $k$. Problems in the form of \eqref{eq1} arise in decentralized multi-agent systems whose tasks are time-varying. Typical applications include adaptive parameter estimation in wireless sensor networks \cite{jakubiec2013d}, decentralized decision-making in dynamic environments \cite{tu2014distributed}, moving-target tracking in multi-agent systems \cite{Rahili2017dynamic}, dynamic resource allocation in communication networks \cite{maros2018beamforming}, and flow control in real-time power systems \cite{tang2019time}, etc. For more applications in decentralized optimization and learning with streaming information, readers are referred to the recent survey paper \cite{emiliano2019survey}.

When the functions $f_i^k(\tilde{x}) \equiv f_i(\tilde{x})$, i.e., they remain unchanged for all times $k$, problem \eqref{eq1} reduces to the decentralized \textit{static deterministic} optimization problem which can be solved by various decentralized methods. There are extensive research works on decentralized algorithms when the true functions $f_i(\tilde{x})$ (or their first-order and second-order information) are fixed and available at all times. In the primal domain, first-order methods such as diffusion \cite{sayed2014adaptive, sayed2014adaptation}, decentralized gradient descent \cite{nedic2009distributed,yuan2016convergence} and dual averaging \cite{duchi2012dual} are effective and easy to implement. Decentralized second-order methods \cite{mokhtari2017network,bajovic2017second} are also able to solve the static problem. However, these primal methods have to employ decaying step-sizes to reach exact convergence to the optimal solution; when constant step-sizes are employed, they will converge to a neighborhood around the optimal solution \cite{nedic2009distributed, sayed2014adaptive, sayed2014adaptation,yuan2016convergence,duchi2012dual,mokhtari2017network,bajovic2017second}. Another family of decentralized algorithms operate in the primal-dual domain, such as those based on the alternating direction method of multipliers (ADMM) \cite{mateos2010distributed,mota2013d,shi2014linear}. By treating the problem from both the primal and dual domains, decentralized ADMM is shown to be able to converge linearly to the exact optimal solution \cite{shi2014linear}, eliminating the limiting bias suffered by the primal methods. However, decentralized ADMM are computationally expensive since it requires each agent to solve a sub-problem at each time. The first-order variants of decentralized ADMM \cite{ling2015dlm,chang2015multi} can alleviate the computational burden by linearizing the local objective functions. Within the family of decentralized primal-dual methods, there are also many approaches that do not explicitly introduce dual variables but can still achieve fast and exact convergence to the optimal solution. These methods include EXTRA \cite{shi2015extra}, exact diffusion \cite{yuan2018exact}, NIDS \cite{li2017decentralized}, and decentralized gradient tracking (DGT) \cite{xu2015augmented,di2016next,nedic2017achieving,qu2018harnessing, xin2018linear}, which have the same computational complexity as the first-order primal methods, but can converge much faster. With all the above results, it is well recognized that the primal-dual methods are superior to the primal ones for decentralized \textit{static deterministic} optimization.

As to another scenario $f_i^k(\tilde{x}) \equiv \mathbb{E}\, Q(\tilde{x}; \xi_i)$ where $Q(\tilde{x},\xi_i)$ is the loss function associated with the optimization variable $\tilde{x}$ and the random variable $\xi_i$, \eqref{eq1} reduces to the decentralized \textit{static stochastic} optimization problem. This formulation is common in decentralized learning applications, where $\xi_i$ represents one or a mini-batch of random data samples at agent $i$. Since the true functions $\mathbb{E}\, Q(\tilde{x}; \xi_i)$ (or their first-order and second-order information) are generally unavailable, one has to use their stochastic approximations in algorithm design. Under this setting, it has been proved in \cite{towfic2015stability} that the primal method diffusion has a wider stability range and better mean-square-error performance than the primal-dual methods such as Arrow-Hurwicz and augmented Lagrangian. However, recent results still come out to endorse the superiority of the primal-dual methods to primal ones. By removing the intrinsic data variance suffered by stochastic diffusion, a primal-dual algorithm called exact diffusion is proved to converge with much smaller mean-square error in the steady state \cite{tang2018d,yuan2019performance}. Moreover, \cite{yuan2019performance} also indicates that the advantage of exact diffusion over stochastic diffusion becomes more evident when the network topology is badly-connected. Similar results can be found in \cite{pu2018distributed,Xin2019introduction}, showing that DGT outperforms diffusion in the context of decentralized \textit{static stochastic} optimization.

While the primal-dual approaches transcend the primal ones in the static deterministic and stochastic problems, to the best of our knowledge, there is no existing work on how these two families of algorithms are compared for the {\em dynamic} problem. Since the dynamic problem \eqref{eq1} can be regarded as a sequence of static ones, can we expect the same conclusion, i.e., the primal-dual methods are superior to the primal ones, in the dynamic scenario? If not, can we clarify conditions under which we should employ the primal methods rather than the primal-dual methods, or vice versa?

Although various primal and primal-dual decentralized dynamic algorithms have been studied in literature, the above questions still remain open with no clear answers. In the primal domain, decentralized dynamic first-order methods proposed in \cite{sun2017distributed,xi2016distributed,liu2017decentralized} can track the dynamic optimal solution $\tilde{x}^{k*}$ with bounded steady-state tracking error when a proper step-size is chosen. Prediction-correction schemes using second-order information are employed in \cite{simonetto2017decentralized} to improve the tracking performance. Primal-dual methods, such as decentralized dynamic ADMM, have also been studied. It is proved that when both $\tilde{x}^{k*}$ and $\nabla f^k_i(\tilde{x}^{k*})$ drift slowly, the decentralized dynamic ADMM is also able to track the dynamic optimal solution \cite{ling2013decentralized}. Other primal-dual methods such as those in \cite{mokhtari2016decentralized,simonetto2014double} reach similar conclusions. Note that all these papers study the primal or primal-dual methods separately and there are no explicit theoretical results on how they compare against each other. It is also difficult to directly compare these algorithms by their bounds of tracking errors shown in \cite{sun2017distributed,xi2016distributed,liu2017decentralized,simonetto2017decentralized,ling2013decentralized,mokhtari2016decentralized,simonetto2014double} since these bounds are derived under different assumptions, and the effects of some important factors such as the network topology are not adequately investigated.

This paper studies and compares the performance of two classical gradient-based methods for decentralized dynamic optimization: one is diffusion and the other is DGT, which are popular primal and primal-dual methods, respectively. 
We establish their convergence properties and show how the drift of optimal solution $\tilde{x}^{k*}$, the drift of optimal gradient $\nabla f_i(\tilde{x}^{k*})$, and the magnitude of optimal gradient $\nabla f_i(\tilde{x}^{k*})$ affect the steady-state tracking performance of both algorithms. In particular, we also explicitly show the influence of network topology on both diffusion and DGT, which, to the best of our knowledge, is the first result to reveal how the network topology affects the steady-state tracking error in decentralized dynamic optimization. With the derived bounds of steady-state tracking errors, we find primal methods {\em can} outperform primal-dual ones and identify conditions under which one family is superior to the other. This sheds lights on how to choose between the primal and primal-dual methods in different decentralized dynamic optimization scenarios.

\subsection{Main results}
To be specific, we will prove in Section \ref{sec:dgd} that with a proper step-size $\alpha = O(1-\beta)$ where $\beta \in (0, 1)$ measures the connectivity of network topology, diffusion converges exponentially fast to a neighborhood around the dynamic optimal solution $\tilde{x}^{k*}$. The steady-state tracking error of diffusion can be characterized as
\begin{align}\label{dgd-bound}
\mbox{diffusion:}\quad\limsup_{k\to \infty} \left(\frac{1}{n}\sum_{i=1}^n\|x_i^k - \tilde{x}^{k*}\|^2\right)^{\frac{1}{2}} \nonumber \\
 = O\left(\frac{\Delta_x}{1-\beta}\right) +  O(\beta D),
\end{align}
where constant $\Delta_x$ measures the drifting rate of $\tilde{x}^{k*}$, and constant $D$ is the upper bound of optimal gradients, i.e., $\|\nabla f_i^k(\tilde{x}^{k*})\| \le D$, $\forall i$. When $\Delta_x = 0$ which corresponds to the scenario of static deterministic optimization, the steady-state performance in \eqref{dgd-bound} reduces to that of static deterministic diffusion \cite{chen2013distributed,sayed2014adaptation,sayed2014adaptive}.

In contrast, we will show in Section \ref{sec:extra} that DGT also converges exponentially fast to a neighborhood around $\tilde{x}^{k*}$, albeit with a smaller step-size $\alpha = O((1-\beta)^2)$. The steady-state tracking error of DGT can be characterized as
\begin{align}\label{gt-bound-3}
\mbox{DGT:}\quad \limsup_{k\to \infty} \left(\frac{1}{n}\sum_{i=1}^n\|x_i^k - \tilde{x}^{k*}\|^2\right)^{\frac{1}{2}} \nonumber \\
 = O\left(\frac{\Delta_x}{(1-\beta)^2}\right) + O(\beta \Delta_g),
\end{align}
where $\Delta_g$ measures the drifting rate of optimal gradients $\nabla f_i^k(\tilde{x}^{k*})$, $\forall i$. When $\Delta_x = 0$ and $\Delta_g = 0$ which corresponds to the scenario of static deterministic optimization, \eqref{gt-bound-3} implies that DGT will converge exactly, which is consistent with the performance of static deterministic DGT \cite{xu2015augmented,di2016next,nedic2017achieving,qu2018harnessing}.

Comparing \eqref{dgd-bound} and \eqref{gt-bound-3}, we observe that while DGT removes the effect of the optimal gradients' upper bound $D$, it incurs a new error related to its drifting rate $\Delta_g$. This result is different from the static (both deterministic and stochastic) scenario  in which the primal-dual approaches completely eliminate the limiting errors suffered by primal methods without introducing any new bias. Further, a badly-connected network with $\beta$ close to $1$ has larger negative effect on DGT than on diffusion. This conclusion is also in contrast to the static stochastic scenario where the primal approaches are more affected by a badly-connected topology than the primal-dual ones. With \eqref{dgd-bound} and \eqref{gt-bound-3}, it is evident that whether diffusion or DGT performs better highly depends on the values of $\Delta_x$, $\Delta_g$, $\beta$ and $D$. {\em The primal approaches can therefore outperform the primal-dual ones in certain scenarios of decentralized dynamic optimization.}

\begin{table*}[t]
	\centering
		\caption{Comparison between diffusion and DGT on steady-state performance for different scenarios.}
	\begin{tabular}{|c|c|c|c|}
		\hline
		\textbf{Scenario}                                    & \textbf{Tracking Error Bound of Diffusion}                  & \textbf{Tracking Error Bound of DGT}                       & \textbf{Better Algorithm} \\ \hline
		$\Delta_x \gg D, \Delta_x \gg \Delta_g$               & $O(\frac{\Delta_x}{1-\beta})$ & $O(\frac{\Delta_x}{(1-\beta)^2})$ & diffusion                       \\ \hline
		$\Delta_x \ll D, \Delta_x \ll \Delta_g, D < \Delta_g$ & $O(D)$                        & $O(\Delta_g)$                     & diffusion                       \\ \hline
		$\Delta_x \ll D, \Delta_x \ll \Delta_g, D > \Delta_g$ & $O(D)$                        & $O(\Delta_g)$                     & DGT                       \\ \hline
	\end{tabular}
	\label{tab:my-table}
\end{table*}

The bounds of steady-state tracking errors given in \eqref{dgd-bound} and \eqref{gt-bound-3}, which we summarize in Table \ref{tab:my-table}, provide guidelines on how to choose between primal and primal-dual methods in  different applications. In the numerical experiments, we will design several delicate examples, in which the quantities $\Delta_x$, $\Delta_g$, $D$ and $\beta$ are controlled, to validate the derived bounds.


\subsection{Other related works}
Dynamic optimization can also be used to formulate the online learning problem, which aims at minimizing a long-term objective in an online manner. For example, the work of \cite{towfic2013distributed} studies the decentralized online classification problem with non-stationary data samples and establishes the convergence property of online diffusion that is similar to the one we derive for dynamic diffusion in Section \ref{sec:dgd}. However, it is assumed in \cite{towfic2013distributed} that the objective functions are twice-differentiable and that the time-varying optimal solution $\tilde{x}^{k*}$ follows a random walk drifting pattern, which are more stringent than the assumptions in this paper. Also, the bound of steady-state tracking error established in \cite{towfic2013distributed} cannot reflect the influence of network topology.

There are some recent primal algorithms that can reach exact convergence for decentralized static optimization by conducting an adaptively increasing number of communication steps per time; see \cite{berahas2018balancing} and \cite{li2018sharp}. However, these methods are not suitable for the dynamic problem, since the introduced multiple inner communication rounds will weaken or even turn off the tracking or adaptation abilities if the dynamic optimal solution changes drastically.

\subsection{Notations}
For a vector $a$, $\|a\|$ stands for the $\ell_2$-norm of $a$. For a matrix $A$, $\|A\|$ and $\rho(A)$ stand for the Frobenius and spectral norms of $A$, respectively. We say $A$ is stable if $\rho(A) < 1$. $\mathds{1}_n \in \R^n$ is the vector with all ones and $I_d \in \R^{d \times d}$ is the identity matrix.



\section{Algorithm Review and Assumptions}
\label{sec:algo}
Consider a bidirectionally connected network of $n$ agents which can communicate with their neighbors. These agents cooperatively solve the decentralized dynamic optimization problem in the form of \eqref{eq1}, with the dynamic versions of diffusion and DGT.

Let $W\in\R^{n\times n}$ be a doubly stochastic matrix, i.e., $W \geq 0$, $W\mathds{1}_n = \mathds{1}_n$ and $W^T\mathds{1}_n = \mathds{1}_n$. Note that $W$ can be non-symmetric. The $(i,j)$-th element $W_{ij}$ is the weight to scale information flowing from agent $j$ to agent $i$. If agents $i$ and $j$ are not neighbors then $W_{ij} = 0$, and if they are neighbors or identical then the weight $W_{ij} \ge 0$. Furthermore, we define $\mathcal{N}_i$ as the set of neighbors of agent $i$ which also includes agent $i$ itself. The diffusion method \cite{chen2013distributed,sayed2014adaptation,sayed2014adaptive} can be used to solve \eqref{eq1} as follows. At time $k+1$, each agent $i$ will conduct
\begin{equation}\label{eq2_1}
x^{k+1}_i=\sum_{j\in \mathcal{N}_i} W_{ij}\big(x_j^k-\A\N f_j^{k+1}(x^k_j)\big), \quad \forall i,
\end{equation}
in parallel, where $x_i$ is the local variable kept by agent $i$. We employ a constant step-size $\alpha$ to keep track of the dynamics of time-varying objective functions. Since $W_{ij} > 0$ holds only for connected agents $i$ and $j$, recursion \eqref{eq2_1} can be implemented in a decentralized manner. The diffusion updates are listed in Algorithm 1. It is expected that each local variable $x_i^k$ will track the dynamic optimal solution $\tilde{x}^{k*}$.

\begin{algorithm}[H]
	\caption{Dynamic diffusion} \hspace*{0.02in} {\bf Input:} Initialize $x_i^0 \in \mathbb{R}^d$, $\forall i$; set $\alpha>0$.
	
	\begin{algorithmic}[1]
		\FOR{$k = 0, 1, \cdots$, every agent $i=1,\cdots,n$}
        \STATE Observe $f_i^{k+1}$ and compute $\phi_i^{k+1} =\hspace{-0.3mm} x_i^{k} \hspace{-0.3mm}-\hspace{-0.3mm} \A\N f_i^{k+1}\hspace{-0.1mm}(\hspace{-0.1mm}x^k_i\hspace{-0.1mm})$
		\STATE Spread $\phi_i^{k+1}$ to and collect $\phi_j^{k+1}$ from neighbors
        \STATE Update local iterate $x^{k+1}_i = \sum_{j\in \mathcal{N}_i} W_{ij} \phi^{k+1}_j$
		\ENDFOR
	\end{algorithmic}
\end{algorithm}

When $f_i^k$ is exactly known and remains unchanged across time, recursion \eqref{eq2_1} reduces to static deterministic diffusion. It has been known that static determinstic diffusion cannot converge exactly to the optimal solution with a constant step-size. Instead, it will converge to a neighborhood around the optimal solution\cite{nedic2009distributed,yuan2016convergence,sayed2014adaptive,sayed2014adaptation,chen2013distributed}. To correct such a steady-state error, one can refer to DGT \cite{xu2015augmented,di2016next,nedic2017achieving,qu2018harnessing}. To solve the decentralized static problem, DGT employs a dynamic average consensus method \cite{zhu2010discrete} to estimate the global gradient and hence removes the limiting bias. Now we adapt it to solve the dynamic problem \eqref{eq1}. In the initialization stage, we let $y_i^0 := \nabla f_i^0(x_i^0)$. At time $k+1$, each agent $i$ will conduct
\begin{align}
x_i^{k+1} &= \sum_{j\in\mathcal{N}_i} W_{ij}\big(x_j^k - \alpha y_j^k\big), \label{diging-11}\\
y_i^{k+1} &= \sum_{j\in\mathcal{N}_i} W_{ij}y_j^k + \nabla f_i^{k+1}(x_i^{k+1}) - \nabla f_i^{k}(x_i^{k}) \label{diging-22}.
\end{align}
In the above recursion, \eqref{diging-22} is called gradient tracking and enables $y_i^k$ to approximate the global gradient asymptotically. The DGT method is listed in Algorithm 2.

We next introduce some notations and common assumptions to facilitate the convergence analysis of dynamic diffusion and DGT. Let $\mathcal{G}(\mathcal{V}, \mathcal{E})$ denote the network where $\mathcal{V}$ is the set of all nodes and $\mathcal{E}$ is the set of all edges. Define $\vx := [x_1;\cdots;x_n]\in \R^{nd}$ be a stack of $x_i \in \R^d$ for $i=1, \cdots,n$, and $\tilde{\vx}^{k*} := [\tilde{x}^{k*};\cdots;\tilde{x}^{k*}]\in \R^{nd}$ be a stack of $\tilde{x}^{k*}\in \R^d$. Also define $F^k(\vx) := \sum_{i=1}^{n}f_i^k(x_i)$. Furthermore, we let $\mathbf{W}:= W \otimes I_d \in \R^{nd \times nd}$ where ``$\otimes$'' indicates the Kronecker product. The following three assumptions are standard in literature.

\begin{algorithm}[t]
	\caption{Dynamic decentralized gradient tracking (DGT)} \hspace*{0.02in}
    {\bf Input:} Initialize $x_i^0 \in \R^d$, $y_i^0 = \N f_i^0(x_i^0)$, $\forall i$; set $\alpha>0$.

	\begin{algorithmic}[1]
		\FOR{$k = 0, 1, \cdots$, every agent $i=1,\cdots,n$}
		\STATE Spread $x^k_i$, $y_i^k$ and collect $x^k_j$, $y^k_j$ from neighbors
		\STATE Observe local objective function $f_i^{k+1}$
		\STATE Update $x^{k+1}_i$ and $y_i^{k+1}$ according to \eqref{diging-11} and \eqref{diging-22}
		\ENDFOR
	\end{algorithmic}
\end{algorithm}

\begin{assumption}[\sc Weight matrix] \label{ass-weight}
	The network is strongly connected and the weight matrix $W\in \R^{n \times n}$ satisfies the following conditions
\begin{align}
	W^T \mathds{1}_n = \mathds{1}_n, \ \ W\mathds{1}_n = \mathds{1}_n, \ \ \mathrm{null}(I - W) = \mathrm{span}(\mathds{1}_n). \nonumber
\end{align}
\end{assumption}
Sort the magnitudes of $W$'s eigenvalues in a decreasing order $|\lambda_1(W)|, |\lambda_2(W)|, \cdots, |\lambda_n(W)|$. Assumption \ref{ass-weight} implies that $1 = |\lambda_1(W)| > |\lambda_2(W)| \ge \cdots \ge |\lambda_n(W)| \geq 0$. In this paper we define $\beta:=|\lambda_2(W)|$.

\begin{assumption}[\sc Smoothness]\label{ass-smooth}
	Each $f_i^k(\tilde{x})$ is convex and has Lipschitz continuous gradients with constant $L_i^k >0$, i.e., it holds for any $\tilde{x}\in \R^d$ and $\tilde{y} \in \R^d$ that
	 \begin{align}\label{ineq-smoothness}
	 	\|\N f_i^k(\tilde{x})-\N f_i^k(\tilde{y})\| \leq L^k_i\|\tilde{x}-\tilde{y}\|, \quad \forall i \in \mathcal{V},\ \forall k.
	 \end{align}
	 Moreover, we assume $L_i^k \le L$ for all $i$ and $k$ where $L>0$ is a constant.
\end{assumption}

\begin{assumption}[\sc Strong convexity] \label{ass-strong}
	Each $f_i^k(\tilde{x})$ is strongly convex with constant $\mu^k_i>0$, i.e., it holds for any $\tilde{x}\in \R^d$ and $\tilde{y} \in \R^d$ that
	\begin{align}
		\LA\N f_i^k(\tilde{x})-\N f_i^k(\tilde{y}),\tilde{x}-\tilde{y}\RA \geq \mu^k_i\|\tilde{x}-\tilde{y}\|^2, \ \forall i \in \mathcal{V},\ \forall k. \nonumber
	\end{align}
	Moreover, we assume $\mu_i^k \ge \mu$ for all $i$ and $k$ where $\mu >0$ is a constant.
\end{assumption}

The following assumption requires that the drift of the dynamic optimal solution is upper-bounded, i.e., $\tilde{x}^{k*}$ changes slowly enough. The constant $\Delta_x$ in Assumption \ref{ass-x-movement} characterizes the drifting rate of $\tilde{x}^{k*}$. Note that we do not assume any specific drifting patterns that $\tilde{x}^{k*}$ has to follow, which is different from \cite{towfic2013distributed} that assumes $\tilde{x}^{k*}$ to follow a random walk model.

\begin{assumption}[\sc Smooth movement of $\tilde{x}^{k*}$] \label{ass-x-movement}
	We assume $\|\tilde{x}^{(k+1)*}-\tilde{x}^{k*}\|\leq \DE_x$ for all times $k=0,1,\cdots$ where $\DE_x > 0$ is a constant.
\end{assumption}

\section{Convergence analysis of dynamic diffusion}
\label{sec:dgd}

In this section we provide convergence analysis for the dynamic diffusion method given in recursion \eqref{eq2_1}. Although there exist results in literature on the convergence of gradient-based dynamic primal methods, these results are either established under more stringent assumptions or ignore the effects of some important influencing factors. We will leave comparisons with the existing results in literature to Remark \ref{remark-dgd-comparison}.

To analyze dynamic diffusion, we will assume the time-varying gradient $\N f_i^k(\tilde{x}^{k*})$ at the \textit{optimal solution} $\tilde{x}^{k*}$ is upper-bounded for any $i$ and $k$. This assumption is much milder than the one used in many existing works (e.g.,  \cite{simonetto2016class,xi2016distributed,simonetto2017decentralized}) that requires the gradient to have a uniform upper bound at \textit{every point} $\tilde{x}$, i.e., $\|\N f_i^k(\tilde{x})\| \le D$.

\begin{assumption}[\sc Boundedness of $\N f_i^k(\tilde{x}^{k*})$] \label{ass-grad-bound}
	We assume $\frac{1}{\sqrt{n}}\sum_{i=1}^n\|\N f_i^k(\tilde{x}^{k*})\| \le D$ for all times $k=0,1,\cdots$ where $D > 0$ is a constant.
\end{assumption}

With the definition of $F(\vx)$ and $\vW$ in Section \ref{sec:algo}, we can rewrite the dynamic diffusion recursion \eqref{eq2_1} as
\begin{align}\label{dynamic-DGD-2}
\vx^{k+1} = \vW \big( \vx^k - \alpha \N F^{k+1}(\vx^k) \big).
\end{align}
By left-multiplying $\frac{1}{n}\mathds{1}^T \otimes I_d$ to both sides of \eqref{dynamic-DGD-2} and defining $\bar{x} := \frac{1}{n} \sum_{i=1}^n x_i = \frac{1}{n}\mathds{1}^T \otimes I_d \vx$, we reach
\begin{align}\label{eq-DGD-avg}
	\bar{x}^{k+1} = \bar{x}^k -  \frac{\alpha}{n}\sum_{i=1}^{n}\nabla f_i^{k+1}(x_i^k).
\end{align}
Define $\bar{\vx}^k:=[\bar{x}^k;\cdots;\bar{x}^k]\in \R^{nd}$. Note that recursion \eqref{eq-DGD-avg} can be rewritten as
\begin{align}\label{238sdghdsb9}
\bar{\vx}^{k+1} = & \bar{\vx}^k -  \frac{\alpha}{n} (\mathds{1}_n\mathds{1}_n^T \otimes I_d)\nabla F^{k+1}(\vx^k) \nonumber \\
                = & \bar{\vx}^k -  \alpha \vR \nabla F^{k+1}(\vx^k),
\end{align}
where $R = \frac{1}{n}\mathds{1}_n\mathds{1}_n^T$ and $\vR = R \otimes I_d \in \mathbb{R}^{dn\times dn}$. It follows that $|\lambda_1(R)|=1$ and $|\lambda_2(R)| = \cdots = |\lambda_n(R)| =0$.

In the following lemmas, we will investigate two distances $\|\bar{\vx}^k - \tilde{\vx}^{k*}\|$ and $\|\vx^k - \bar{\vx}^k\|$ to facilitate analyzing the convergence of dynamic diffusion. Here  $\tilde{\vx}^{k*} := [\tilde{x}^{k*};\cdots;\tilde{x}^{k*}]\in \R^{nd}$ as we have defined in Section \ref{sec:algo}.

\begin{lemma}\label{lm-DGD-inequality-1}
Under Assumptions \ref{ass-smooth}--\ref{ass-x-movement}, if step-size $\alpha\le \frac{2}{\mu + L}$, it holds that
\begin{align}\label{s234hds8-2}
\|\bar{\vx}^{k+1} - \tilde{\vx}^{(k+1)*}\| &\le \Big(1 - \frac{\alpha \mu}{2}\Big) \|\bar{\vx}^k  \hspace{-0.8mm}-\hspace{-0.8mm} \tilde{\vx}^{k*}\| \nonumber \\
&\hspace{-1cm} + \alpha L \| \vx^k - \bar{\vx}^k \| + \left(1 - \frac{\alpha \mu}{2}\right) \sqrt{n} \Delta_x.
\end{align}
\end{lemma}

\begin{proof} See Appendix \ref{app-1}.
\end{proof}

\begin{lemma}\label{lm-DGD-inequality-2}
Under Assumptions \ref{ass-weight}, \ref{ass-smooth}, \ref{ass-x-movement} and \ref{ass-grad-bound}, if step-size $\alpha \le \frac{2}{\mu + L}$, it holds that
\begin{align}\label{238sdbm20-0}
\|\vx^{k+1} - \bar{\vx}^{k+1}\| &{\le} \beta \|\vx^k - \bar{\vx}^k\| + \alpha \beta L \|\bar{\vx}^k - \tilde{\vx}^{k*}\| \nonumber \\
& + \alpha \beta L \sqrt{n} \Delta_x + \alpha \beta \sqrt{n} D.
\end{align}
\end{lemma}

\begin{proof} See Appendix \ref{app-2}.

\end{proof}

With Lemmas \ref{lm-DGD-inequality-1} and \ref{lm-DGD-inequality-2}, if $\alpha \le \frac{2}{\mu + L}$, we have
\begin{align}\label{DGD-main-recursion}
\hspace{-0.5em}\underbrace{\ba{c}
\|\bar{\vx}^{k+1} - \tilde{\vx}^{(k+1)*}\| \\
\|\vx^{k+1} - \bar{\vx}^{k+1}\|
\ea}_{:= {\boldsymbol z}^{k+1}} &\le \underbrace{\ba{cc}
	1 - \frac{\alpha \mu}{2} & \alpha L \\
	\alpha \beta L & \beta
	\ea}_{:=A}
\underbrace{\ba{c}
\|\bar{\vx}^{k} - \tilde{\vx}^{k*}\| \\
\|\vx^{k} - \bar{\vx}^{k}\|
\ea}_{:= {\boldsymbol z}^k} \nonumber \\
& +
\underbrace{\ba{c}
\left(1 - \frac{\alpha \mu}{2}\right) \sqrt{n} \Delta_x \\
\alpha \beta L \sqrt{n} \Delta_x + \alpha \beta \sqrt{n} D
\ea}_{:= {\boldsymbol b}}.
\end{align}
In the following lemma, we examine $\rho(A)$, the spectral norm of matrix $A$. With it, we reach the main result in Theorem \ref{thm-1}.

\begin{lemma}\label{lm-rho-A-diffusion}
When step-size $\alpha \le \frac{\mu(1-\beta)}{10L^2} = O(\frac{\mu(1-\beta)}{L^2})$, it holds that $\rho(A) \leq 1- \frac{3\mu \alpha}{8} = 1 - O(\mu \alpha) < 1$.
\end{lemma}
\begin{proof} See Appendix \ref{app-3}.
\end{proof}


\begin{theorem}\label{thm-1}
Under Assumptions \ref{ass-weight}--\ref{ass-grad-bound}, if step-size $\alpha \le \frac{\mu(1-\beta)}{10L^2} = O(\frac{\mu(1-\beta)}{L^2})$, it holds that the variable $\vx^k$ generated by dynamic diffusion recursion \eqref{dynamic-DGD-2} will converge exponentially fast, at rate $\rho(A) = 1 - O(\mu \alpha)$, to a neighborhood around $\tilde{\vx}^{k*}$. Moreover, the steady-state tracking error is bounded as
\begin{align}\label{eq-thm-dgd-error}
\limsup_{k\to \infty} \frac{1}{\sqrt{n}}\|\vx^k - \tilde{\vx}^{k*}\| \hspace{-1mm} \le  \hspace{-1mm} \left( \frac{4}{\alpha \mu}  \hspace{-0.8mm}+ \hspace{-0.8mm} \frac{4\beta L}{\mu(1-\beta)} \right)\Delta_x  \hspace{-0.8mm}+ \hspace{-0.8mm} \frac{6\alpha\beta L D}{\mu(1-\beta)}.
\end{align}
\end{theorem}
\begin{proof} See Appendix \ref{app-4}.
\end{proof}

\begin{remark} From inequality \eqref{eq-thm-dgd-error}, we further have
\begin{align}\label{dgd-steady-state-error}
 &\hspace{-1cm} \limsup_{k\to \infty} \left(\frac{1}{n}\sum_{i=1}^n\|x_i^k - \tilde{x}^{k*}\|^2\right)^{\frac{1}{2}}\nonumber \\
  =&\ O\left( \frac{\Delta_x}{\alpha} + \frac{\beta \Delta_x}{1-\beta} \right) +\ O\left(\frac{\alpha \beta D}{1-\beta}\right),
\end{align}
where we ignore the influences of constants $L$ and $\mu$.
This result implies that when the step-size $\alpha$ is sufficiently small, dynamic diffusion cannot effectively track the dynamic optimal solution $\tilde{\vx}^{k*}$ because of the term $O(\frac{\Delta_x}{\alpha})$, which is consistent with our intuition. On the other hand, when $\alpha$ is too large, the inherent bias term $O(\frac{\alpha \beta D}{1-\beta})$ will dominate and deteriorate the steady-state tracking performance.
\end{remark}

\begin{remark}
	If $f^k_i(\tilde{x}) \equiv f_i(\tilde{x})$ remains unchanged with time, we have $\Delta_x = 0$ and $\frac{1}{\sqrt{n}}\sum_{i=1}^n\|\nabla f_i(\tilde{x}^*)\| \le D$. In this scenario, inequality \eqref{eq-thm-dgd-error} implies
	\begin{align}
		\limsup_{k\to \infty} \left(\frac{1}{n}\sum_{i=1}^n\|x_i^k - \tilde{x}^{k*}\|^2\right)^{\frac{1}{2}} = O\left( \frac{\alpha \beta D}{1-\beta} \right),
	\end{align}
	which is consistent with the convergence property for static diffusion, as derived in \cite{chen2013distributed,sayed2014adaptation,sayed2014adaptive}.
\end{remark}

\begin{remark}\label{remark-dgd-comparison}
We compare the result in \eqref{dgd-steady-state-error} with the existing results for gradient-based decentralized dynamic primal algorithms in literature. First, the results in \cite{popkov2005gradient,towfic2013distributed,sun2017distributed,liu2017decentralized,simonetto2016class,simonetto2017decentralized} are established for either quadratic or twice-differentiable objective functions. While \cite{xi2016distributed} provides convergence analysis for first-order differentiable objective functions as this work, it has to assume the gradients are uniformly upper-bounded for all iterates. In comparison, we only assume that the optimal gradients are upper-bounded. Second, the bounds of steady-state tracking error derived in some of the existing works ignore the effects of certain influencing factors and are hence less precise. For example, \cite{sun2017distributed,xi2016distributed,liu2017decentralized,simonetto2016class,simonetto2017decentralized,popkov2005gradient} do not distinguish the dynamics-dependent tracking error $O(\frac{\Delta_x}{\alpha})$ from the intrinsic bias $O(\frac{\alpha \beta D}{1-\beta})$, which exists even for the decentralized static problem. Moreover, the influence of network topology is also ignored in these works.
\end{remark}

\begin{remark}
If the network is fully connected and $W = \frac{1}{n}\mathds{1}\mathds{1}^T$, it holds that $\beta=0$ and hence the bound given in \eqref{dgd-steady-state-error} reduces to
$
\limsup_{k\to \infty} \left(\frac{1}{n}\sum_{i=1}^n\|x_i^k - \tilde{x}^{k*}\|^2\right)^{\frac{1}{2}}=O\left( \frac{\Delta_x}{\alpha} \right),
$
which matches with the performance of centralized gradient descent for dynamic optimization. However, the bounds of decentralized dynamic gradient descent derived in \cite{sun2017distributed,xi2016distributed,liu2017decentralized} are worse than the bound of centralized gradient descent even if $\beta=0$.
\end{remark}

\section{Convergence Analysis of Dynamic DGT}
\label{sec:extra}

Decentralized gradient tracking (DGT) is a popular primal-dual approach proposed for decentralized static optimization. In this section we analyze the convergence property of its dynamic variant given in \eqref{diging-11}--\eqref{diging-22}. Instead of assuming the boundedness of $\nabla f_i^k(\tilde{x}^{k*})$ as in Assumption \ref{ass-grad-bound}, we assume that the movement of $\nabla f_i^k(\tilde{x}^{k*})$ is smooth for dynamic DGT.

\begin{assumption}[\sc Smooth movement of $\nabla f_i^k(\tilde{x}^{k*})$]\label{ass-g-movement}
	We assume $\frac{1}{\sqrt{n}}\sum_{i=1}^n\|\nabla f_i^{k+1}(\TI{x}^{(k+1)*}) - \nabla f_i^{k}(\TI{x}^{k*})\|\leq \DE_g$ for all times $k=0,1,\cdots$ where $\DE_g > 0$ is a constant.
\end{assumption}

We introduce $\vy := [y_1;\cdots;y_n]\in \mathbb{R}^{nd}$. With the definition of $F(\vx)$ in Section \ref{sec:algo}, we can rewrite recursion \eqref{diging-11}--\eqref{diging-22} as
\begin{align}
\vx^{k+1} &= \vW(\vx^k - \alpha \vy^k), \label{diging-1}\\
\vy^{k+1} &= \vW \vy^k + \nabla F^{k+1}(\vx^{k+1}) - \nabla F^{k}(\vx^{k}), \label{diging-2}
\end{align}
where $\vy^0 = \N F^0(\vx^0)$. In the following lemmas, we will use three quantities $\|\bar{\vx}^k - \tilde{\vx}^{k*}\|$, $\|{\vx}^k - \bar{\vx}^k\|$ and $\|{\vy}^k - \bar{\vy}^k\|$ (where $\bar{\vy}^k := [\bar{y}^k; \cdots; \bar{y}^k]\in \mathbb{R}^{nd}$ and $\bar{y}^k:=\frac{1}{n}\sum_{i=1}^{n}y_i^k$) to establish the convergence of dynamic DGT recursion \eqref{diging-1}--\eqref{diging-2}. To this end, we recall from Section \ref{sec:dgd} that $R=\frac{1}{n}\mathds{1}_n \mathds{1}_n^T$ and $\vR = R\otimes I_d \in \mathbb{R}^{dn\times dn}$. By left-multiplying $\vR$ to both sides of recursion \eqref{diging-2} we have
\begin{align} \label{23bs9987}
\bar{\vy}^{k+1} = \bar{\vy}^{k} + \vR\big(\nabla F^{k+1}(\vx^{k+1}) - \nabla F^{k}(\vx^{k})\big).
\end{align}
Since $\vy^0 = \nabla F^0(\vx^0)$ and hence $\bar{\vy}^0 = \vR \nabla F^0(\vx^0)$, for $k=1, 2, \cdots$ we can reach
\begin{align} \label{tttaaa}
\bar{\vy}^k = \vR\nabla F^{k}(\vx^k).
\end{align}

The following three lemmas characterize the evolution of $\|\vy^k - \bar{\vy}^k\|$, $\|\vx^{k} - \bar{\vx}^k\|$ and $\|\bar{\vx}^k - \tilde{\vx}^{k*}\|$, respectively.

\begin{lemma}\label{lm-diging-y-bary}
Under Assumptions \ref{ass-weight}, \ref{ass-smooth}, \ref{ass-x-movement} and \ref{ass-g-movement}, if step-size $\alpha \le \frac{1-\beta}{2L}$, it holds that
\begin{align}\label{bound-y}
\|\vy^{k+1} - \bar{\vy}^{k+1}\| &\le \frac{1+\beta}{2}\|\vy^k - \bar{\vy}^k\| + 5 L\|\vx^k - \bar{\vx}^k\| \nonumber \\
&\quad + 3 L\|\bar{\vx}^k \hspace{-1mm}-\hspace{-1mm} \tilde{\vx}^{k*}\|^2 \hspace{-1mm}+\hspace{-1mm}  L\sqrt{n}\Delta_x \hspace{-1mm}+\hspace{-1mm}  \sqrt{n}  \Delta_g.
\end{align}
\end{lemma}
\begin{proof} See Appendix \ref{app-5}.
\end{proof}

\begin{lemma}\label{lm-bound-x-xbar} Under Assumption \ref{ass-weight}, it holds that
\begin{align}\label{bound-x-xbar}
\|{\vx}^{k+1} - \bar{\vx}^{k+1}\| \le \beta \|\vx^k - \bar{\vx}^k\| + \alpha \beta \|\vy^k - \bar{\vy}^k\|.
\end{align}
\end{lemma}
\begin{proof} See Appendix \ref{app-6}.
\end{proof}

\begin{lemma}\label{lm-xbar-xstar}
Under Assumptions \ref{ass-weight}--\ref{ass-x-movement}, if step-size $\alpha \le \frac{2}{\mu + L}$ then it follows that
\begin{align}\label{lm6-eq}
\|\bar{\vx}^{k+1} - \tilde{\vx}^{(k+1)*}\| \le&\ \left( 1 - \frac{\alpha \mu}{2} \right)\|\bar{\vx}^k - \tilde{\vx}^{k*}\| \nonumber \\
&\ + \alpha L\|\bar{\vx}^k - {\vx}^{k}\| + \sqrt{n}\Delta_x.
\end{align}
\end{lemma}
\begin{proof} See Appendix \ref{app-7}.
\end{proof}

With Lemmas \ref{lm-diging-y-bary}, \ref{lm-bound-x-xbar} and \ref{lm-xbar-xstar}, when $\alpha \le \min\{\frac{1-\beta}{2L}, \frac{2}{\mu + L}\}$ we reach an inequality as
\begin{align}\label{23nsd8}
& \underbrace{\ba{c}
\|\vy^{k+1} - \bar{\vy}^{k+1}\| \\
\|{\vx}^{k+1} - \bar{\vx}^{k+1}\| \\
\|\bar{\vx}^{k+1} - \tilde{\vx}^{(k+1)*}\|
\ea}_{:={\boldsymbol z}^{k+1}} \nonumber \\
\le &\
\underbrace{\ba{ccc}
\frac{1+\beta}{2} & 5 L & 3 L \\
\alpha \beta & \beta & 0 \\
0 & \alpha L & 1 - \frac{\alpha \mu}{2}
\ea}_{:=A}
\underbrace{\ba{c}
\|\vy^{k} - \bar{\vy}^{k}\| \\
\|{\vx}^{k} - \bar{\vx}^{k}\| \\
\|\bar{\vx}^{k} - \tilde{\vx}^{k*}\|
\ea}_{:={\boldsymbol z}^{k}} \nonumber \\
&\ +
\underbrace{\ba{c}
L\sqrt{n}\Delta_x + \sqrt{n} \Delta_g \\
0 \\
\sqrt{n}\Delta_x
\ea}_{:= {\boldsymbol b}}.
\end{align}
The following lemma shows that when step-size $\alpha$ is sufficiently small, the matrix $A$ is stable, i.e., $\rho(A) < 1$.
\begin{lemma}\label{lm-7}
	When step-size $\alpha \le \frac{(1-\beta)^2\mu}{c L^2} = O(\frac{(1-\beta)^2\mu}{L^2})$ where $c$ is a constant independent of $\beta$, $\mu$ and $L$, it holds that $\rho(A) \le 1 - \frac{\alpha \mu}{4} = 1 - O(\mu \alpha) < 1$. Moreover, the matrix $I-A$ is invertible and it follows that
	\begin{align}
	&\ (I-A)^{-1} \nonumber \\
	\le&\ \frac{8}{(1-\beta)^2\alpha \mu} \ba{ccc}
	\frac{\alpha \mu (1-\beta)}{2} & 6\alpha L^2 & 3 L(1-\beta)\\
	\frac{\alpha^2 \beta \mu}{2} & \frac{\alpha \mu (1-\beta)}{4} & 3\alpha \beta L \\
	\alpha^2 \beta L & \frac{\alpha L(1-\beta)}{2} & \frac{(1-\beta)^2}{2}
	\ea.
	\end{align}
\end{lemma}
\begin{proof} See Appendix \ref{app-8}.
\end{proof}

Finally, we bound the steady-state tracking error of dynamic DGT in the following theorem.
\begin{theorem}\label{th-2}
	Under Assumptions \ref{ass-weight}--\ref{ass-x-movement} and \ref{ass-g-movement}, if step-size $\alpha \le \frac{(1-\beta)^2\mu}{c L^2} = O(\frac{(1-\beta)^2\mu}{L^2})$ where $c$ is a constant independent of $\beta$, $\mu$ and $L$, it holds that the variable $\vx^k$ generated by the dynamic DGT recursion \eqref{diging-1}--\eqref{diging-2} will converge exponentially fast, at rate $\rho(A) = 1 - O(\mu \alpha)$, to a neighborhood around $\tilde{\vx}^{k*}$. Moreover, the steady-state tracking error is bounded as
	\begin{align}\label{eq-thm-diging}
	&\ \limsup_{k\to \infty} \frac{1}{\sqrt{n}}\|{\vx}^{k} - \tilde{\vx}^{k*}\| \nonumber \\
	\le&\ \left( \frac{4}{\alpha \mu} + \frac{40\beta L}{(1-\beta)^2 \mu} \right) \Delta_x  + \frac{16\alpha \beta L}{(1-\beta)^2\mu} \Delta_g.
	\end{align}
\end{theorem}
\begin{proof} See Appendix \ref{app-9}.
\end{proof}

\begin{remark} From inequality \eqref{eq-thm-diging}, we further have
	\begin{align}\label{gt-steady-state-error}
 &\hspace{-1cm} \limsup_{k\to \infty} \left(\frac{1}{n}\sum_{i=1}^n\|x_i^k - \tilde{x}^{k*}\|^2\right)^{\frac{1}{2}} \nonumber \\
=&\ O\left( \frac{\Delta_x}{\alpha} + \frac{\beta \Delta_x}{(1-\beta)^2} \right) + O\left(\frac{\alpha \beta \Delta_g}{(1-\beta)^2}\right),
	\end{align}
	where we ignore the influences of constants $L$ and $\mu$.
	This result implies that when step-size $\alpha$ is sufficiently small, dynamic DGT cannot effectively track the dynamic optimal solution $\tilde{\vx}^{k*}$ because of the term $O(\frac{\Delta_x}{\alpha})$, which is consistent with our intuition and similar to dynamic diffusion. On the other hand, when $\alpha$ is too large, the inherent bias term $O(\frac{\alpha \beta \Delta_g}{(1-\beta)^2})$ will dominate and deteriorate the steady-state tracking performance.
\end{remark}

\begin{remark}
	If $f^k_i(\tilde{x}) \equiv f_i(\tilde{x})$ remains unchanged with time, we have $\Delta_x = 0$ and $\Delta_g=0$. In this scenario, inequality \eqref{eq-thm-dgd-error} implies
	\begin{align}
	\limsup_{k\to \infty} \left(\frac{1}{n}\sum_{i=1}^n\|x_i^k - \tilde{x}^{k*}\|^2\right)^{\frac{1}{2}} = 0,
	\end{align}
	which is consistent with the convergence property for static DGT, as derived in \cite{xu2015augmented,di2016next,nedic2017achieving, qu2018harnessing}.
\end{remark}

\begin{remark}\label{remark-gt-comparison}
	We compare the result in \eqref{dgd-steady-state-error} with the existing results for decentralized dynamic primal-dual methods in literature. It is observed from \eqref{dgd-steady-state-error} that the drifts of both optimal solutions and optimal gradients, i.e., $\Delta_x$ and $\Delta_g$, will affect the steady-state tracking performance of dynamic DGT. This is consistent with the results for dynamic ADMM \cite{ling2013decentralized} and primal-descent dual-ascent methods \cite{mokhtari2016decentralized,simonetto2014double}. There are two major differences between the bound in \eqref{gt-steady-state-error} and those derived in \cite{ling2013decentralized,mokhtari2016decentralized,simonetto2014double}. First, the bound in \eqref{gt-steady-state-error} distinguishes the tracking error caused by the drift of optimal solutions and that caused by the drift of optimal gradients, while \cite{ling2013decentralized,mokhtari2016decentralized,simonetto2014double} mixed all error terms into one. For this reason, it is difficult to conclude from the results in \cite{ling2013decentralized,mokhtari2016decentralized,simonetto2014double} that the tracking error caused by the drift of optimal solutions will dominate when the step-size is sufficiently small. Second, all the bounds derived in \cite{ling2013decentralized,mokhtari2016decentralized,simonetto2014double} ignore the influence of network topology which, as we have shown in \eqref{dgd-steady-state-error} and will validate in the numerical experiments, is a key component to clarify why primal methods can outperform primal-dual methods in certain scenarios.
\end{remark}

\begin{remark}
If the network is fully connected and $W = \frac{1}{n}\mathds{1}\mathds{1}^T$, it holds that $\beta=0$ and hence the bound given in \eqref{gt-steady-state-error} reduces to
	$
	\limsup_{k\to \infty} \left(\frac{1}{n}\sum_{i=1}^n\|x_i^k - \tilde{x}^{k*}\|^2\right)^{\frac{1}{2}}=O\left( \frac{\Delta_x}{\alpha} \right),
	$
	which, similar to dynamic diffusion, matches with the performance of centralized gradient descent for dynamic optimization.
\end{remark}

\section{Comparison Between Dynamic Diffusion and Dynamic DGT}
\label{sec: comparison}
In this section we compare the steady-state tracking performance between dynamic diffusion and dynamic DGT, and discuss scenarios in which one method can outperform the other. By substituting the established step-sizes $\alpha_{\rm diffusion} = O(1-\beta)$ and $\alpha_{\rm DGT} = O((1-\beta)^2)$ into the bounds of steady-state tracking errors in \eqref{dgd-steady-state-error} and \eqref{gt-steady-state-error}, respectively, we reach
\begin{align}
\mbox{diffusion:}\quad \limsup_{k\to \infty}& \left(\frac{1}{n}\sum_{i=1}^n\|x_i^k - \tilde{x}^{k*}\|^2\right)^{\frac{1}{2}} \nonumber \\
=&\ O\left(\frac{\Delta_x}{1-\beta} \right) + O\left(\beta D\right), \label{primal-error} \\
\mbox{DGT:}\quad \limsup_{k\to \infty}& \left(\frac{1}{n}\sum_{i=1}^n\|x_i^k - \tilde{x}^{k*}\|^2\right)^{\frac{1}{2}} \nonumber \\
=&\ O\left(\frac{\Delta_x}{(1-\beta)^2} \right) + O\left(\beta \Delta_g\right). \label{primal-dual-error}
\end{align}
Comparing \eqref{primal-error} and \eqref{primal-dual-error}, we observe that while DGT removes the effect of the averaged magnitude of optimal gradients (i.e., the quantity $D$), it introduces a new error incurred by the drift of the optimal gradients (i.e., the quantity $\Delta_g$). Furthermore, the error term $\Delta_x$ is magnified by $O(\frac{1}{(1-\beta)^2})$ for DGT, which is worse than the primal method, diffusion. These two facts are very different from the previous results for decentralized static optimization, in which the primal-dual approaches completely remove $D$ without bringing in any new error and suffer less from badly-connected network topologies.

Bounds \eqref{primal-error} and \eqref{primal-dual-error} imply that either primal or primal-dual approaches can be superior to the other for different values of $\Delta_x$, $\Delta_g$, $D$ and $\beta$. Appropriate algorithms should be carefully chosen for different scenarios. We summarize the guidelines for choosing algorithms as follows.
\begin{itemize}
	\item When the drifting rate of optimal solution dominates, i.e., $\Delta_x \gg D$ and $\Delta_x \gg \Delta_g$, diffusion has smaller steady-state tracking error than DGT. Moreover, the worse the topology is (i.e., the closer $\beta$ is to $1$), the more evident the advantage of diffusion over DGT is. In this scenario, one should choose diffusion rather than DGT.
	
	\item When $\Delta_x \ll D$, $\Delta_x \ll \Delta_g$, $D < \Delta_g$, and the network is well-connected such that $\beta$ is not close to $1$, diffusion has smaller steady-state tracking error than DGT. In this scenario, one should choose diffusion rather than DGT.
	
	\item When $\Delta_x \ll D$, $\Delta_x \ll \Delta_g$, $\Delta_g < D$, and the network is well-connected such that $\beta$ is not close to $1$, DGT has smaller steady-state tracking error than diffusion. In this scenario, one should choose DGT rather than diffusion.
\end{itemize}

Next we construct several examples in which we can control $\Delta_x$, $\Delta_g$, $D$ and $\beta$. With these examples, we will validate the above guidelines via numerical experiments in Section \ref{sec:exp}.
	
\subsection{Scenario I: $D = \Delta_g = 0$ and $\Delta_x > 0$}
We consider a decentralized dynamic least-squares problem in the form of
\begin{align}
\min_{\tilde{x}\in \mathbb{R}^d}\quad \frac{1}{2}\sum_{i=1}^{n} \|C_i^k \tilde{x} - r_i^k \|^2.
\end{align}
The dynamic optimal solution $\tilde{x}^{k*}$ is generated following a trajectory such as circle or sinusoid and drifts slowly, with $\Delta_x > 0$. At time $k$, the coefficient matrix $C_i^k \in \mathbb{R}^{m\times d}$ is observed, where $m < d$ but $mn > d$. Then, the measurement vector is given by $r_i^k = C_i^k \tilde{x}^{k*}$. For this example, we can verify that
$
\nabla f_i^k(\tilde{x}^{k*}) = (C_i^k)^T(C_i^k \tilde{x}^{k*} - r_i^k) = 0, \forall i\in \mathcal{V}, \forall k=1,2,\cdots,
$
and hence it holds that
\begin{align}
D &= \max_k\{\frac{1}{\sqrt{n}}\sum_{i=1}^{n}\|\nabla f_i^k(\tilde{x}^{k*})\|\} = 0, \nonumber \\
\Delta_g &= \max_k\{\frac{1}{\sqrt{n}}\sum_{i=1}^n\|\nabla f_i^{k+1}(\TI{x}^{(k+1)*}) - \nabla f_i^{k}(\TI{x}^{k*})\|\}=0. \nonumber
\end{align}
To verify the effect of network topology, we consider the linear and cyclic graphs in which $1-\beta = O(\frac{1}{n^2})$, and the grid graph in which $1-\beta = O(\frac{1}{n})$ \cite{yuan2019performance}. By varying the network size $n$, we can adjust the value of $1-\beta$.

\subsection{Scenario II: $\Delta_x = 0$ and $\Delta_g > D > 0$}\label{sec:sce-2}
We consider a dynamic average consensus problem in the form of
\begin{align}\label{ex_1}
\min\limits_{\TI{x}\in \R} \quad \frac{1}{2}\sum_{i=1}^n (\tilde{x} - y_i^k)^2,
\end{align}
where $y_i^0 = i\cdot m$ and $m$ is a given positive constant. For simplicity we assume $n = 2p+1$ where $p$ is a positive integer. The optimal solution at time $0$ is given by
\begin{align}\label{ex_1_sol}
\TI{x}^{0*}=\F{1}{n}\sum_{i=1}^n y_i^0 = \left(\frac{n+1}{2}\right) m = (p+1)m.
\end{align}
For each time $k$, we conduct circular shifting for the sequences $\{y_i^k\}_{i=1}^n$, such that
\begin{align}
y_i^{k+1} = y_{\langle i-T \rangle_n}^{k},
\end{align}
where the operator $\langle i-T \rangle_n = i-T ~ \text{mod} ~ n$ if $i-T ~ \text{mod} ~ n \neq 0$ and $\langle i-T \rangle_n = n$ if $i-T ~ \text{mod} ~ n = 0$, while $T$ is a given constant. Apparently, since only the locations of $\{y_i^k\}_{i=1}^n$ vary, we conclude that $\TI{x}^{k*}\equiv (p+1)m$ for any $k$ and hence $\Delta_x = 0$. On the other hand, note that
\begin{align}\label{23nsdnds9}
 \frac{1}{\sqrt{n}}\sum_{i=1}^{n}|\nabla f^k_i(\tilde{x}^{k*})| =&\ \frac{1}{\sqrt{n}}\sum_{i=1}^{n}|\tilde{x}^{k*} - y_i^k| \nonumber \\
=&\  \frac{1}{\sqrt{2p+1}}\sum_{i=1}^{2p+1}|p+1 - i|m \nonumber \\
=&\ \frac{p(p+1)m}{\sqrt{2p+1}},\quad \forall\, k=1,2,\cdots,
\end{align}
which implies that $D=p(p+1)m/\sqrt{2p+1}$. In addition, we can examine
\begin{align}
&\ |\nabla f^k_i(\tilde{x}^{k*}) - \nabla f^{k-1}_i(\tilde{x}^{(k-1)*})| \nonumber \\
=&\ |\nabla f^k_i(\tilde{x}^{k*}) - \nabla f^{k-1}_i(\tilde{x}^{k*})| = |y_i^k - y_i^{k-1}|,
\end{align}
which implies that $\Delta_g = \max_{k}\{\frac{1}{\sqrt{n}}\sum_{i=1}^{n}|y_i^k - y_i^{k-1}|\}$. By setting $T=p+1$, we then have
\begin{align}\label{23dshsd}
\Delta_g = \frac{2p(p+1)m}{\sqrt{n}} = \frac{2p(p+1)m}{\sqrt{2p+1}}.
\end{align}
Comparing \eqref{23nsdnds9} and \eqref{23dshsd}, we have that $D < \Delta_g$ and the gap $\Delta_g - D$ becomes larger as $p$ or $m$ increases.

\subsection{Scenario III: $\Delta_x = 0$ and $D > \Delta_g > 0$}
We still consider the setting in scenario II, but let $T=1$ and hence reach
\begin{align}\label{23dshsd-2}
\Delta_g = \frac{2(n-1)m}{\sqrt{n}} = \frac{4pm}{\sqrt{2p+1}}.
\end{align}
Comparing \eqref{23nsdnds9} and \eqref{23dshsd-2}, we have that $\Delta_g < D$ when $p>3$ and the gap $D - \Delta_g$ becomes larger as $p$ or $m$ increases.

\begin{figure}[t]
	\centering
	\includegraphics[scale=0.45]{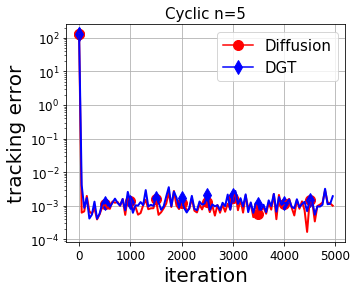}
	\includegraphics[scale=0.45]{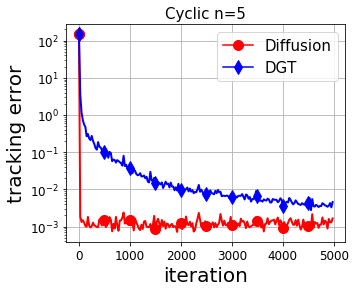}
	\includegraphics[scale=0.45]{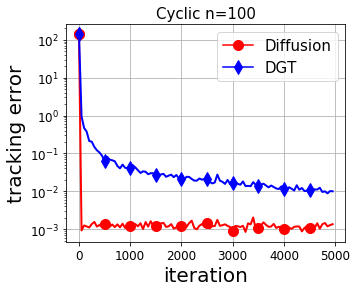}
	\caption{Scenario I: Comparison between dynamic diffusion and DGT over cyclic graphs with different sizes: $n=5$, $n=50$, and $n=100$.}
	\label{fig:different_size}
\end{figure}

\section{Numerical Experiments}
\label{sec:exp}

In this section we compare the performance of the two dynamic methods, diffusion and DGT, in the scenarios discussed in Section \ref{sec: comparison}.

\subsection{Scenario I: $D = \Delta_g = 0$ and $\Delta_x > 0$}
We consider the example discussed in Scenario I. Assume that $\tilde{x}^{k*}\in \mathbb{R}^2$ moves along a circle as
$$
\tilde{x}^{k*}(1) = \cos\left(\frac{3\pi k}{2 M} \right), \quad \tilde{x}^{k*}(2) = \sin\left(\frac{3\pi k}{2 M}\right),
$$
where $\tilde{x}^{k*}(1)$ and $\tilde{x}^{k*}(2)$ denote the first and second coordinates of $\tilde{x}^{k*}$, respectively. We let $M=5000$, $k\in [0, M-1]$, $m=1$, and generate $C_i^k \sim \mathcal{N}(0, 1) \in \mathbb{R}^{1\times 2} $ and $r_i^k = C_i^k \tilde{x}^{k*} \in \mathbb{R}$. We compare dynamic diffusion and DGT over cyclic graphs with sizes $5$, $50$ and $100$ in Fig. \ref{fig:different_size}. The step-sizes for both algorithms are tuned to optimal by hand to reach the best steady-state tracking errors. The $y$-axis indicates the tracking error, defined as $(\frac{1}{n}\sum_{i=1}^n\|x_i^k - \tilde{x}^{k*}\|^2)^{1/2}/\|\tilde{x}^{k*}\|^2$. Note that $\|\tilde{x}^{k*}\|^2$ is set as time-invariant in all the numerical experiments. It is observed that when $n=5$ with $\beta=0.54$, diffusion is slightly better than DGT. For $n=50$ with $\beta=0.994$, diffusion outperforms DGT. For $n=100$ with $\beta=0.999$, diffusion is observed significantly better than DGT. The phenomenon shown in Fig. \ref{fig:different_size} validates our result that DGT is more sensitive to badly-connected topologies.

Next we compare the primal method diffusion with the dynamic versions of other primal-dual algorithms, including EXTRA \cite{shi2015extra}, exact-diffusion (E-diffusion) \cite{yuan2017exact1}, ADMM \cite{ling2013decentralized}, and DLM \cite{ling2015dlm}. We consider a cyclic graph with $n=100$ and $\beta=0.999$, and tune the parameters (e.g., the step-size, the augmented Lagrangian coefficient, etc) to the optimal so that each algorithm reaches its best steady-state tracking error; see Fig. \ref{fig:track_circle}. Note that diffusion is still superior to the others, which implies that our analysis on DGT can be potentially generalized to other primal-dual algorithms and that primal methods can outperform primal-dual methods when $\Delta_x \gg D$ and $\Delta_x \gg \Delta_g$ and $\beta$ is close to $1$.

\begin{figure}[t]
	\centering
	\includegraphics[scale=0.45]{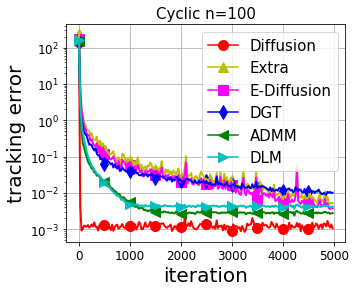}
	\caption{Scenario I: Comparison between different dynamic algorithms over a cyclic graph with size $n=100$.}
	\label{fig:track_circle}
\end{figure}

\subsection{Scenario II: $\Delta_x = 0$ and $\Delta_g > D > 0$}
We consider the example discussed in Scenario II. Set the parameters as $m=1$, $p=1000$ so that $n=2p+1 = 2001$, and $T=p+1=1001$ (the definition of $T$ is given in Section \ref{sec:sce-2}). We exploit a random network with $n=2001$ and $\beta=0.89$. Since $D < \Delta_g$ (see \eqref{23nsdnds9} and \eqref{23dshsd}), it is expected that diffusion will outperform DGT in terms of the steady-state tracking error; see the discussion in Section \ref{sec:sce-2}. In Fig. \ref{fig:average_1}, we compare dynamic diffusion with the dynamic versions of DGT, EXTRA, exact diffusion (E-diffusion), ADMM, and DLM. The parameters of all algorithms are tuned to the optimal to reach the best steady-state tracking errors. It is observed that diffusion is still superior to all the primal-dual methods including DGT. This implies that primal methods can outperform primal-dual methods when $D \gg \Delta_x$, $\Delta_g \gg \Delta_x$, and $\Delta_g > D$.

\subsection{Scenario III: $\Delta_x = 0$ and $D > \Delta_g > 0$}
Now we consider the example discussed in Scenario III and a random network with $n=2001$ and $\beta=0.89$. We set $m=1$, $p=1000$ so that $n=2p+1 = 2001$, and $T=1$. In Fig. \ref{fig:average_2}, we compare dynamic diffusion with the dynamic versions of DGT, EXTRA, exact diffusion (E-diffusion), ADMM, and DLM. The parameters of all algorithm are tuned to the optimal to reach the best steady-state tracking errors. DGT has the lowest steady-state tracking error among all the primal-dual methods and diffusion has the worse performance. This implies that primal-dual methods can outperform primal methods when $D \gg \Delta_x$, $\Delta_g \gg \Delta_x$, and $\Delta_g < D$.

\begin{figure}[H]
	\centering
	\includegraphics[scale=0.45]{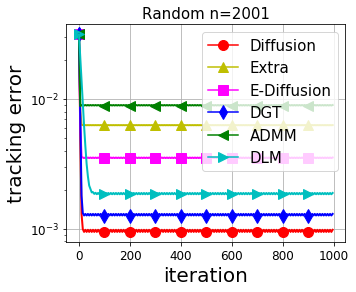}
	\caption{Scenario II: Comparison between different dynamic algorithms over a random network with size $n=2001$ and $\beta=0.89$.}
	\label{fig:average_1}
\end{figure}

\begin{figure}[H]
	\centering
	\includegraphics[scale=0.45]{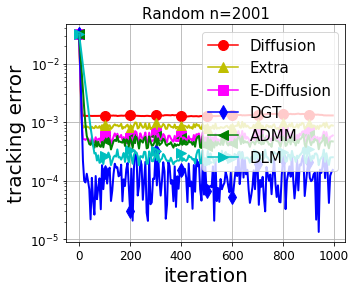}
	\caption{Scenario III: Comparison between different dynamic algorithms over a random network with size $n=2001$ and $\beta=0.89$.}
	\label{fig:average_2}
\end{figure}

\section{Conclusion}
\label{sec:conclusion}
This paper investigates the primal and primal-dual methods in solving the decentralized dynamic optimization problem. By studying two representative methods, i.e., diffusion and decentralized gradient tracking (DGT), we find that primal algorithms can outperform primal-dual algorithms under certain conditions. In particular, we prove that diffusion is greatly affected by the magnitudes of dynamic optimal gradients, while DGT is more sensitive to the drifts of dynamic optimal gradients. Theoretical analysis also shows that diffusion works better over badly-connected network. These conclusions provide guidelines on how to choose proper dynamic algorithms in different scenarios.


\appendices
\section{Proof of Lemma \ref{lm-DGD-inequality-1}}\label{app-1}
From recursion \eqref{eq-DGD-avg} we have
\begin{align}
\bar{x}^{k+1} - \tilde{x}^{(k+1)*} &= \bar{x}^k - \tilde{x}^{(k+1)*} -  \frac{\alpha}{n}\sum_{i=1}^{n}\nabla f_i^{k+1}(\bar{x}^k)\nonumber \\
& -  \frac{\alpha}{n}\sum_{i=1}^{n}[\nabla f_i^{k+1}(x_i^k) - \nabla f_i^{k+1}(\bar{x}^k)].
\end{align}
Using the triangle inequality, we have
\begin{align}\label{s234hds8}
\|\bar{x}^{k+1} - \tilde{x}^{(k+1)*}\| &\le \|\bar{x}^k \hspace{-0.8mm}-\hspace{-0.8mm} \tilde{x}^{(k+1)*} \hspace{-0.8mm}-\hspace{-0.8mm}  \frac{\alpha}{n}\sum_{i=1}^{n}\nabla f_i^{k+1}(\bar{x}^k)\| \nonumber\\
&\hspace{-1cm}+ \alpha \|  \frac{1}{n}\sum_{i=1}^{n}[\nabla f_i^{k+1}(x_i^k) - \nabla f_i^{k+1}(\bar{x}^k)]\|.
\end{align}

Now we check the first term at the right-hand side of \eqref{s234hds8}. Note that
\begin{align}\label{xbwe78}
&\ \|\bar{x}^k \hspace{-0.8mm}-\hspace{-0.8mm} \tilde{x}^{(k+1)*} \hspace{-1mm}-\hspace{-1mm}  \frac{\alpha}{n}\sum_{i=1}^{n}\nabla f_i^{k+1}(\bar{x}^k)\|^2 \nonumber \\
\overset{(a)}{=}& \|\bar{x}^k \hspace{-1mm}-\hspace{-1mm} \tilde{x}^{(k+1)*} \hspace{-1mm}-\hspace{-1mm}  \big(\frac{\alpha}{n}\hspace{-1mm}\sum_{i=1}^{n}\nabla f_i^{k+1}(\bar{x}^k) \hspace{-1mm}-\hspace{-1mm}  \frac{\alpha}{n}\hspace{-1mm}\sum_{i=1}^{n}\nabla f_i^{k+1}(\tilde{x}^{(k+1)*}) \big) \|^2 \nonumber \\
\overset{(b)}{\le}& \Big(1 - \frac{2\alpha \mu L}{\mu + L}\Big)\|\bar{x}^k \hspace{-0.8mm}-\hspace{-0.8mm} \tilde{x}^{(k+1)*}\|^2 \nonumber \\
&-\hspace{-1mm} \Big(\hspace{-0.5mm}\frac{2\alpha}{\mu + L} \hspace{-0.5mm}-\hspace{-0.5mm} \alpha^2\hspace{-0.5mm}\Big)\hspace{-0.3mm}\|\frac{1}{n}\hspace{-1mm}\sum_{i=1}^{n}\hspace{-0.5mm}\nabla f_i^{k+1}(\bar{x}^k) \hspace{-0.5mm}-\hspace{-0.5mm} \frac{1}{n}\hspace{-1mm}\sum_{i=1}^{n}\hspace{-0.6mm}\nabla f_i^{k+1}(\tilde{x}^{(k+1)*})\|^2 \nonumber \\
\overset{(c)}{\le}&\ \Big(1 - \frac{2\alpha \mu L}{\mu + L}\Big)\|\bar{x}^k \hspace{-0.8mm}-\hspace{-0.8mm} \tilde{x}^{(k+1)*}\|^2 \nonumber\\
\overset{(d)}{\le}&\ \Big( 1 - \frac{\alpha \mu L}{\mu + L}\Big)^2 \|\bar{x}^k \hspace{-0.8mm}-\hspace{-0.8mm} \tilde{x}^{(k+1)*}\|^2\nonumber \\
\overset{(e)}{\le}&\ \Big( 1 - \frac{\alpha \mu}{2}\Big)^2 \|\bar{x}^k \hspace{-0.8mm}-\hspace{-0.8mm} \tilde{x}^{(k+1)*}\|^2.
\end{align}
Here equality $(a)$ holds since $\frac{1}{n}\sum_{i=1}^{n}\nabla f_i^{k+1}(\tilde{x}^{(k+1)*}) = 0$. Inequality $(b)$ holds because
\begin{align}\label{23n33n}
&\ \langle \bar{x}^k \hspace{-0.8mm}-\hspace{-0.8mm} \tilde{x}^{(k+1)*}, \frac{1}{n}\sum_{i=1}^{n}\nabla f_i^{k+1}(\bar{x}^k) - \frac{1}{n}\sum_{i=1}^{n}\nabla f_i^{k+1}(\tilde{x}^{(k+1)*})  \rangle \nonumber \\
\ge&\ \frac{1}{\mu + L} \|\frac{1}{n}\sum_{i=1}^{n}\nabla f_i^{k+1}(\bar{x}^k) - \frac{1}{n}\sum_{i=1}^{n}\nabla f_i^{k+1}(\tilde{x}^{(k+1)*})  \|^2 \nonumber \\
&\ + \frac{\mu L}{\mu + L} \|\bar{x}^k \hspace{-0.8mm}-\hspace{-0.8mm} \tilde{x}^{(k+1)*}\|^2,
\end{align}
which is adapted from Theorem 2.1.12 in \cite{nesterov2004introductory} and $\mu$, $L$ are introduced in Assumptions \ref{ass-smooth} and \ref{ass-strong}. Inequality $(c)$ holds since we set $\alpha \le \frac{2}{\mu + L}$ such that $\frac{2\alpha}{\mu + L} - \alpha^2 \geq 0$. Inequality $(d)$ holds because $1-\frac{2\alpha \mu L}{\mu + L} \le 1 - \frac{2\alpha \mu L}{\mu + L} + \Big(\frac{\alpha \mu L}{\mu + L}\Big)^2 = \Big(1 - \frac{\alpha \mu L}{\mu + L}\Big)^2$ and $(e)$ holds because $2L \ge L + \mu$ and hence $1 - \frac{\alpha \mu L}{\mu + L} \le 1 - \frac{\alpha \mu}{2}$ (note that $1 - \frac{\alpha \mu L}{\mu + L} > 0$ when $\alpha \le \frac{2}{\mu + L}$). As a result, when $\alpha \le \frac{2}{\mu + L}$, it holds that
\begin{align}\label{pweudsh}
&\|\bar{x}^k \hspace{-0.8mm}-\hspace{-0.8mm} \tilde{x}^{(k+1)*} \hspace{-0.8mm}-\hspace{-0.8mm}  \frac{\alpha}{n}\sum_{i=1}^{n}\nabla f_i^{k+1}(\bar{x}^k)\| \nonumber\\
&\le \left( 1 - \frac{\alpha \mu}{2} \right) \| \bar{x}^k \hspace{-0.8mm}-\hspace{-0.8mm} \tilde{x}^{(k+1)*} \| \nonumber\\
&\le \left( 1 - \frac{\alpha \mu}{2} \right) \| \bar{x}^k \hspace{-0.8mm}-\hspace{-0.8mm} \tilde{x}^{k*} \| + \left( 1 - \frac{\alpha \mu}{2} \right) \Delta_x,
\end{align}
where the last inequality holds because $\|\bar{x}^k \hspace{-0.8mm}-\hspace{-0.8mm} \tilde{x}^{(k+1)*}\| \le \|\bar{x}^k \hspace{-0.8mm}-\hspace{-0.8mm} \tilde{x}^{k*}\| + \|\tilde{x}^{k*} \hspace{-0.8mm}-\hspace{-0.8mm} \tilde{x}^{(k+1)*}\|$ and $\|\tilde{x}^{k*} \hspace{-0.8mm}-\hspace{-0.8mm} \tilde{x}^{(k+1)*}\|\le \Delta_x$ (see Assumption \ref{ass-x-movement}).

For the second term at the right-hand side of \eqref{s234hds8}, we have
\begin{align}\label{sbwed8}
& \|  \frac{1}{n}\sum_{i=1}^{n}[\nabla f_i^{k+1}(x_i^k) - \nabla f_i^{k+1}(\bar{x}^k)]\| \nonumber\\
& \le \frac{1}{n}\sum_{i=1}^{n}\|\nabla f_i^{k+1}(x_i^k) - \nabla f_i^{k+1}(\bar{x}^k)\| \nonumber \\
& \overset{(a)}{\le} \frac{L}{n}\sum_{i=1}^{n}\|x_i^k - \bar{x}^k\| \overset{(b)}{\le} \frac{L}{\sqrt{n}}\|\vx^k - \bar{\vx}^k\|,
\end{align}
where inequality $(a)$ holds because of Assumption \ref{ass-smooth}, and $(b)$ holds since $\sum_{i=1}^{n}\|x_i^k - \bar{x}^k\| \le \sqrt{n} \|\vx^k - \bar{\vx}^k\|$. By substituting \eqref{pweudsh} and \eqref{sbwed8} into \eqref{s234hds8}, we obtain
\begin{align}
\|\bar{x}^{k+1} - \tilde{x}^{(k+1)*}\| &\le \Big(1 - \frac{\alpha \mu}{2}\Big) \|\bar{x}^k \hspace{-0.8mm}-\hspace{-0.8mm} \tilde{x}^{k*}\| \nonumber\\
& \hspace{-0.8mm}+\hspace{-0.8mm} \frac{\alpha L}{\sqrt{n}} \| \vx^k - \bar{\vx}^k \| +  \left( 1 - \frac{\alpha \mu}{2} \right) \Delta_x.
\end{align}
Recalling that $\|\bar{\vx}^{k+1} - \tilde{\vx}^{(k+1)*}\| = \sqrt{n}\|\bar{x}^{k+1} - \tilde{x}^{(k+1)*}\|$, we obtain the inequality in \eqref{s234hds8-2}.

\section{Proof of Lemma \ref{lm-DGD-inequality-2}}\label{app-2}
By subtracting \eqref{238sdghdsb9} from \eqref{dynamic-DGD-2}, we reach
\begin{align}\label{238sdbm20}
\vx^{k+1} - \bar{\vx}^{k+1} &= \vW \vx^k - \bar{\vx}^k - \alpha \left( \vW - \vR \right) \nabla F^{k+1}(\vx^k) \nonumber \\
&\overset{(a)}{=} (\vW \hspace{-0.6mm}-\hspace{-0.6mm} \vR) (\vx^k \hspace{-0.6mm}-\hspace{-0.6mm} \bar{\vx}^k) \hspace{-0.6mm}-\hspace{-0.6mm} \alpha \left( \vW \hspace{-0.6mm}-\hspace{-0.6mm} \vR \right) \nabla F^{k+1}(\vx^k) \nonumber \\
&= \left(\vW \hspace{-0.5mm}-\hspace{-0.5mm} \vR\right) \big(\vx^k \hspace{-0.5mm}-\hspace{-0.5mm} \bar{\vx}^k - \alpha [\nabla F^{k+1}(\vx^k)  \nonumber\\
&\hspace{1cm} - \nabla F^{k+1}(\bar{\vx}^k)] - \alpha\nabla F^{k+1}(\bar{\vx}^{k}) \big),
\end{align}
where $(a)$ holds because $\bar{\vx}^k = \vR \vx^k$ and $(\vW - \vR)\bar{\vx}^k=0$. By taking $\ell_2$-norm on both sides of the above equality, we reach
\begin{align}
&\ \|\vx^{k+1} - \bar{\vx}^{k+1}\|\nonumber\\
&\le \hspace{-1mm} \rho(\hspace{-0.3mm}\vW \hspace{-1mm}-\hspace{-1mm} \vR) \hspace{-0.3mm} \|\vx^k \hspace{-1mm}-\hspace{-1mm} \bar{\vx}^k \hspace{-1mm}-\hspace{-1mm} \alpha\nabla F^{k+1}(\hspace{-0.3mm}\bar{\vx}^{k}\hspace{-0.3mm})\nonumber \\
& \hspace{2cm} - \alpha [\nabla F^{k+1}(\hspace{-0.3mm}\vx^k\hspace{-0.3mm}) \hspace{-1mm}-\hspace{-1mm} \nabla F^{k+1}(\hspace{-0.3mm}\bar{\vx}^{k}\hspace{-0.3mm})] \hspace{-1mm}\| \nonumber \\
&\overset{(a)}{\le} \beta \|\vx^k \hspace{-1mm}-\hspace{-1mm} \bar{\vx}^k\hspace{-1mm}-\hspace{-1mm}\alpha [\nabla F^{k+1}(\vx^k) \hspace{-1mm}-\hspace{-1mm} \nabla F^{k+1}(\bar{\vx}^{k})]\hspace{-1mm}-\hspace{-1mm} \alpha\nabla F^{k+1}(\bar{\vx}^{k})\| \nonumber \\
&\overset{(b)}{\le} \beta \|\vx^k \hspace{-0.5mm}-\hspace{-0.5mm} \bar{\vx}^k - \alpha [\nabla F^{k+1}(\vx^k) - \nabla F^{k+1}(\bar{\vx}^{k})]\| \nonumber \\
&\quad + \alpha \beta L \|\bar{\vx}^{k} - \tilde{\vx}^{k*}\| \nonumber \\
&\quad + \alpha \beta L \|\tilde{\vx}^{k*} - \tilde{\vx}^{(k+1)*}\| \nonumber \\
&\quad + \alpha \beta \|\nabla F^{k+1}(\tilde{\vx}^{(k+1)*})\| \nonumber \\
&\overset{(c)}{\le} \beta \|\vx^k \hspace{-0.5mm}-\hspace{-0.5mm} \bar{\vx}^k\| \hspace{-1mm}+\hspace{-1mm} \alpha \beta L \|\bar{\vx}^k \hspace{-1mm}-\hspace{-1mm} \tilde{\vx}^{k*}\| \hspace{-1mm}+\hspace{-1mm} \alpha \beta L \sqrt{n}\Delta_x \hspace{-1mm}+\hspace{-1mm} \alpha \beta \sqrt{n} D, \nonumber
\end{align}
where $(a)$ holds because $\rho(\vW - \vR) = \beta :=|\lambda_2(W)|$, $(b)$ holds because
\begin{align*}
&\|\nabla F^{k+1}(\bar{\vx}^{k})\| \le \|\nabla F^{k+1}(\bar{\vx}^{k}) \hspace{-0.8mm}-\hspace{-0.8mm} \nabla F^{k+1}(\tilde{\vx}^{k*})\| \nonumber \\
&\quad + \|\nabla F^{k+1}(\tilde{\vx}^{k*}) \hspace{-0.8mm}-\hspace{-0.8mm} \nabla F^{k+1}(\tilde{\vx}^{(k+1)*})\| \hspace{-0.8mm}+\hspace{-0.8mm} \|\nabla F^{k+1}(\tilde{\vx}^{(k+1)*})\| \nonumber \\
& \le L \|\bar{\vx}^{k} - \tilde{\vx}^{k*}\| + L \|\tilde{\vx}^{k*} - \tilde{\vx}^{(k+1)*}\|+ F^{k+1}(\tilde{\vx}^{(k+1)*})\|,
\end{align*}
and $(c)$ holds because $\|\vx^k \hspace{-0.5mm}-\hspace{-0.5mm} \bar{\vx}^k - \alpha [\nabla F^{k+1}(\vx^k) - \nabla F^{k+1}(\bar{\vx}^{k})]\|\le(1-\frac{\alpha \mu}{2})\|\vx^k \hspace{-0.5mm}-\hspace{-0.5mm} \bar{\vx}^k\| \le \|\vx^k \hspace{-0.5mm}-\hspace{-0.5mm} \bar{\vx}^k\|$ when $\alpha \le \frac{2}{\mu+L}$ (which can be derived by following the arguments in \eqref{xbwe78} and \eqref{23n33n}), and
\begin{align*}
\| \tilde{\vx}^{k*} \hspace{-0.8mm}-\hspace{-0.8mm} \tilde{\vx}^{(k+1)*}\|^2 &= \sum_{i=1}^n \|\tilde{x}^{k*} - \tilde{x}^{(k+1)*}\|^2 \le n \Delta_x^2,  \\
\|\nabla F^{k+1}(\tilde{\vx}^{(k+1)*})\|^2 &= \sum_{i=1}^{n}\|\nabla f_i^{k+1}(\tilde{x}^{(k+1)*})\|^2 \\
&\le \Big(\sum_{i=1}^{n} \|\nabla f_i^{k+1}(\tilde{x}^{(k+1)*})\|\Big)^2 \le n D^2.
\end{align*}
This concludes the proof.

\section{Proof of Lemma \ref{lm-rho-A-diffusion}}\label{app-3}
The characteristic polynomial of matrix $A$ is given by $$p(\tau) = \tau^2 - \Big( 1 - \frac{\alpha \mu}{2} + \beta \Big) \tau + \beta \Big(1 - \frac{\alpha \mu}{2}\Big) - \beta \alpha^2 L^2.$$
It can be verified that the roots $\tau_1$ and $\tau_2$ (we assume $\tau_1 > \tau_2$) of $p(\tau)=0$ are given by
\begin{align}
\tau = \frac{\big( 1 - \frac{\alpha \mu}{2} + \beta \big) \pm \sqrt{\big( 1 - \frac{\alpha \mu}{2} - \beta \big)^2 + 4 \beta \alpha^2 L^2}}{2}.
\end{align}
If $ 1 - \frac{\alpha \mu}{2} + \beta > 0$, the spectral norm of matrix $A$ is determined by $\tau_1$, i.e., $\rho(A) = \max\{|\tau_1|, |\tau_2|\} = \tau_1$. When $\alpha \le \frac{\mu(1-\beta)}{10L^2}$, we can prove that $1 - \frac{\alpha \mu}{2} + \beta > 0$ and
\begin{align}
\label{eq-DGD-root-bound}
\big( 1 - \frac{\alpha \mu}{2} - \beta \big)^2 + 4\beta \alpha^2 L^2 \le \big( 1 - \frac{\alpha \mu}{4} - \beta \big)^2.
\end{align}
To see so, with simple algebraic operations, one can verify that inequality \eqref{eq-DGD-root-bound} is equivalent to
\begin{align}
\label{2dbd99}
4\alpha \beta L^2 + \frac{3}{16} \alpha \mu^2 \le \frac{\mu}{2}(1-\beta).
\end{align}
Since $4\alpha \beta L^2 + \frac{3}{16} \alpha \mu^2 \le 4\alpha L^2 + \frac{3}{16} \alpha L^2 \le 5 \alpha L^2$, it is enough to set $\alpha \le \frac{\mu(1-\beta)}{10 L^2}$ to guarantee \eqref{2dbd99}. Inequality \eqref{eq-DGD-root-bound} implies that
$\rho(A) = \tau_1 \le [1 - \frac{\alpha \mu}{2} + \beta + 1 - \frac{\alpha \mu}{4} - \beta]/2= 1 - \frac{3\alpha \mu}{8} < 1$,
which completes the proof.

\section{Proof of Theorem \ref{thm-1}}\label{app-4}
Since $\rho(A)<1$, we know $I-A$ is invertible. Then, the inequality ${\boldsymbol z}^k \le A {\boldsymbol z}^{k-1} + {\boldsymbol b}$ in \eqref{DGD-main-recursion} leads to
\begin{align}
{\boldsymbol z}^k \le (A)^k {\boldsymbol z}^{0} + (I-A)^{-1}{\boldsymbol b}.
\end{align}
Since $\rho(A) = 1 - O(\mu\alpha) < 1$, it holds from Theorems  8.5.1  and  8.5.2 of \cite{horn2012matrix} that each entry in $(A)^k$ will vanish with rate $(\rho(A))^k$. Therefore, both $\|\bar{\vx}^{k} - \tilde{\vx}^{k*}\|$ and $\|\vx^{k} - \bar{\vx}^{k}\|$ will also converge with bounded error at rate $(\rho(A))^k$, and the limiting error is
\begin{align}
\limsup_{k\to \infty} {\boldsymbol z}^k = (I-A)^{-1}{\boldsymbol b}.
\end{align}

Next we examine $(I-A)^{-1}$. Note that
\begin{align}\label{23bbs9}
(I-A)^{-1} &=
\ba{cc}
\frac{\alpha \mu}{2} & -\alpha L \\
-\alpha \beta L & 1-\beta
\ea^{-1} \nonumber \\
&=
\frac{1}{\frac{(1-\beta)\alpha \mu }{2} - \alpha^2 L^2 \beta}
\ba{cc}
1-\beta & \alpha  L \\
\alpha \beta L & \frac{\alpha \mu}{2}
\ea \nonumber \\
&\overset{(a)}{\le}
\ba{cc}
\frac{4}{\alpha \mu} & \frac{4L}{\mu(1-\beta)}\\
\frac{4\beta L}{\mu(1-\beta)} & \frac{2}{1-\beta}
\ea,
\end{align}
where $(a)$ holds because $\frac{(1-\beta)\alpha \mu }{2} - \alpha^2 L^2 \beta \ge \frac{(1-\beta)\alpha \mu }{4}$ when $\alpha \le \frac{\mu(1-\beta)}{10 L^2}$. With \eqref{23bbs9} we have
\begin{align}
\limsup_{k\to \infty} {\boldsymbol z}^k \hspace{-1mm}\leq\hspace{-1mm}  \ba{cc}
\frac{4}{\alpha \mu} & \frac{4L}{\mu(1-\beta)}\\
\frac{4\beta L}{\mu(1-\beta)} & \frac{2}{1-\beta}
\ea
\ba{c}
\left(1 - \frac{\alpha \mu}{2}\right) \sqrt{n} \Delta_x \\
\sqrt{n}\alpha \beta L \Delta_x + \sqrt{n}\alpha \beta D
\ea, \nonumber
\end{align}
which further implies the following two inequalities
\begin{align}
&\hspace{-1em}\limsup_{k\to \infty} \|\bar{\vx}^{k} - \tilde{\vx}^{k*}\| \nonumber \\
&\le \frac{4\sqrt{n}\Delta_x}{\alpha \mu} - 2\sqrt{n}\Delta_x + \frac{4\sqrt{n}\alpha \beta L^2 \Delta_x}{\mu(1-\beta)} + \frac{4\sqrt{n}\alpha \beta L D}{\mu(1-\beta)} \nonumber \\
& \overset{(a)}{\le}  \frac{4\sqrt{n}\Delta_x}{\alpha \mu} + \frac{4\sqrt{n}\alpha \beta L D}{\mu(1-\beta)},
\end{align}
\begin{align}
&\limsup_{k\to \infty} \|\vx^{k} - \bar{\vx}^{k}\| \nonumber \\
&\le \frac{4\beta L\sqrt{n}\Delta_x}{\mu(1-\beta)} \hspace{-0.3mm}-\hspace{-0.3mm} \frac{2\alpha \beta L \sqrt{n} \Delta_x}{1-\beta} \hspace{-0.3mm}+\hspace{-0.3mm} \frac{2\alpha \beta L \sqrt{n}\Delta_x}{1-\beta} \hspace{-0.4mm}+\hspace{-0.4mm} \frac{2\sqrt{n}\alpha \beta D}{1-\beta} \nonumber \\
&\hspace{-0.4mm}=\hspace{-0.4mm} \frac{4\beta L\sqrt{n}\Delta_x}{\mu(1-\beta)} \hspace{-0.4mm}+\hspace{-0.4mm} \frac{2\sqrt{n}\alpha \beta D}{1-\beta} \nonumber \\
&\hspace{-0.4mm}\le\hspace{-0.4mm}\hspace{-0.4mm} \frac{4\beta L\sqrt{n}\Delta_x}{\mu(1-\beta)} \hspace{-0.4mm}+\hspace{-0.4mm} \frac{2\sqrt{n}\alpha \beta L D}{\mu(1-\beta)}.
\end{align}
Here $(a)$ holds because $- 2\sqrt{n}\Delta_x + \frac{4\sqrt{n}\alpha \beta L^2 \Delta_x}{\mu(1-\beta)} < 0$ when $\alpha \le \frac{\mu(1-\beta)}{10 L^2}$. With the above two inequalities, we have
\begin{align}
&\limsup_{k\to \infty} \frac{1}{\sqrt{n}} \|\vx^{k} - \tilde{\vx}^{k*}\| \nonumber \\
\le&\ \frac{1}{\sqrt{n}} \limsup_{k\to \infty} \|\bar{\vx}^{k} - \tilde{\vx}^{k*}\| + \frac{1}{\sqrt{n}} \limsup_{k\to \infty} \|\vx^{k} - \bar{\vx}^{k}\| \nonumber \\
\le&\ \left( \frac{4}{\alpha \mu} + \frac{4\beta L}{\mu(1-\beta)} \right)\Delta_x + \frac{6\alpha \beta L D}{\mu(1-\beta)},
\end{align}
which completes the proof.

\section{Proof of Lemma \ref{lm-diging-y-bary}}\label{app-5}
By subtracting \eqref{23bs9987} from \eqref{diging-2}, we have
\begin{align} \label{23bs99}
& \vy^{k+1} - \bar{\vy}^{k+1} \\
=& (\vW - \vR)\vy^k + (I_{nd} - \vR)\big(\nabla F^{k+1}(\vx^{k+1}) - \nabla F^{k}(\vx^{k})\big) \nonumber \\
\overset{(a)}{=}& (\vW \hspace{-1mm}-\hspace{-1mm} \vR)(\vy^k \hspace{-1mm}-\hspace{-1mm} \bar{\vy}^k) \hspace{-1mm}+\hspace{-1mm} (I_{nd} \hspace{-1mm}-\hspace{-1mm} \vR)\big(\nabla F^{k+1}(\vx^{k+1}) \hspace{-1mm}-\hspace{-1mm} \nabla F^{k}(\vx^{k})\big), \nonumber
\end{align}
where $(a)$ holds because $\vW\bar{\vy}^k = \vR\bar{\vy}^k$. Since $\rho(\vW - \vR) = \beta$ and $\rho(I_{nd}-\vR) = 1$, from \eqref{23bs99} we have
\begin{align}
&\ \|\vy^{k+1} - \bar{\vy}^{k+1}\| \nonumber \\
\le& \beta\|\vy^k - \bar{\vy}^k\| + \|\nabla F^{k+1}(\vx^{k+1}) - \nabla F^{k}(\vx^{k})\| \nonumber \\
\le& \beta \|\vy^k - \bar{\vy}^k\| + \|\nabla F^{k+1}(\vx^{k+1}) - \nabla F^{k+1}(\tilde{\vx}^{(k+1)*})\| \nonumber  \\
&\hspace{-4mm} +\hspace{-0.8mm} \|\hspace{-0.3mm}\nabla\hspace{-0.6mm} F^{k+1}\hspace{-0.4mm}(\tilde{\vx}^{(k+1)*}\hspace{-0.4mm}) \hspace{-0.8mm}-\hspace{-0.8mm} \nabla \hspace{-0.6mm} F^{k}\hspace{-0.4mm}(\tilde{\vx}^{k*})\hspace{-0.4mm}\| \hspace{-0.8mm}+\hspace{-0.8mm} \|\nabla F^{k}(\tilde{\vx}^{k*}) \hspace{-0.8mm}-\hspace{-0.8mm} \nabla F^{k}(\vx^{k}) \| \nonumber \\
\overset{(a)}{\le}& \beta \|\vy^k - \bar{\vy}^k\| \hspace{-1mm} + \hspace{-1mm} L\|\vx^{k+1} \hspace{-1mm}-\hspace{-1mm} \tilde{\vx}^{(k+1)*}\| \hspace{-1mm}+\hspace{-1mm} \sqrt{n} \Delta_g \hspace{-1mm}+\hspace{-1mm} L\|\vx^{k} - \tilde{\vx}^{k*}\|  \nonumber \\
\overset{(b)}{\le}& \beta \|\vy^k - \bar{\vy}^k\| \hspace{-0.4mm}+\hspace{-0.4mm} L\|\vx^{k+1} \hspace{-0.4mm}-\hspace{-0.4mm} \vx^{k}\| \nonumber \\
& \hspace{-0.4mm}+\hspace{-0.4mm}  2L\|\vx^k - \bar{\vx}^k\| \hspace{-0.6mm}+\hspace{-0.6mm} 2L\|\bar{\vx}^k \hspace{-0.6mm}-\hspace{-0.6mm} \tilde{\vx}^{k*}\| \hspace{-0.6mm}+\hspace{-0.6mm} L\sqrt{n}\Delta_x \hspace{-0.6mm}+\hspace{-0.6mm}  \sqrt{n} \Delta_g. \label{6sd8n}
\end{align}
Inequality $(a)$ holds because of \eqref{ineq-smoothness} in Assumption \ref{ass-smooth} and
\begin{align}
&\ \|\nabla F^{k+1}(\tilde{\vx}^{(k+1)*}) - \nabla F^{k}(\tilde{\vx}^{k*})\|^2 \nonumber \\
\le&\ \left( \sum_{i=1}^{n} \|\nabla f_i^{k+1}(\tilde{x}^{(k+1)*}) - \nabla f_i^{k}(\tilde{x}^{k*}) \| \right)^2 \le n \Delta_g^2,
\end{align}
where the last inequality holds because of Assumption \ref{ass-g-movement}. Inequality $(b)$ holds because
\begin{align}
\|\vx^{k+1} - \tilde{\vx}^{(k+1)*}\| &\le \|\vx^{k+1} - \vx^{k}\| + \|\vx^{k} - \bar{\vx}^k\| \nonumber \\
&+ \|\bar{\vx}^k - \tilde{\vx}^{k*}\| + \|\tilde{\vx}^{k*} - \tilde{\vx}^{(k+1)*} \|, \\
\|\vx^k - \tilde{\vx}^{k*}\| &\le \|\vx^{k} - \bar{\vx}^k\| + \|\bar{\vx}^k - \tilde{\vx}^{k*}\|,
\end{align}
and $\|\tilde{\vx}^{k*} - \tilde{\vx}^{(k+1)*} \| \le \sqrt{n}\Delta_x$.

Next, we turn to bounding the term $\|\vx^{k+1} - \vx^{k}\|$. With recursion \eqref{diging-1}, we have
\begin{align}
&\vx^{k+1} - \vx^k \nonumber\\
=& (\vW - I_{nd})\vx^k - \alpha \vW \vy^k \nonumber\\
=& (\vW - I_{nd})\vx^k - \alpha \vW (\vy^k - \bar{\vy}^k + \bar{\vy}^k - \vR\nabla F^k(\tilde{\vx}^{k*})) \nonumber \\
=& (\vW \hspace{-0.5mm}-\hspace{-0.5mm} I_{nd})(\vx^k \hspace{-0.5mm}-\hspace{-0.5mm} \bar{\vx}^k) \hspace{-0.5mm}-\hspace{-0.5mm} \alpha \vW (\vy^k \hspace{-0.5mm}-\hspace{-0.5mm} \bar{\vy}^k \hspace{-0.5mm}+\hspace{-0.5mm} \bar{\vy}^k \hspace{-0.5mm}-\hspace{-0.5mm} \vR\nabla F^k(\tilde{\vx}^{k*})). \nonumber
\end{align}
The above relation leads to
\begin{align}
&\ \|\vx^{k+1} - \vx^k\| \nonumber \\
\overset{(a)}{\le}& 2 \|\vx^k \hspace{-0.5mm}-\hspace{-0.5mm} \bar{\vx}^k\| \hspace{-0.5mm}+\hspace{-0.5mm} \alpha \|\vy^k \hspace{-0.5mm}-\hspace{-0.5mm} \bar{\vy}^k\| \hspace{-0.5mm}+\hspace{-0.5mm} \alpha \|\vR \nabla F^k(\vx^k) \hspace{-0.5mm}-\hspace{-0.5mm} \vR \nabla F^k(\tilde{\vx}^{k*})\| \nonumber \\
\overset{(b)}{\le}& 2 \|\vx^k - \bar{\vx}^k\| + \alpha \|\vy^k - \bar{\vy}^k\| + \alpha L \|\vx^k - \tilde{\vx}^{k*}\| \nonumber\\
\le& (2 \hspace{-0.5mm}+\hspace{-0.5mm} \alpha L)\|\vx^k \hspace{-0.5mm}-\hspace{-0.5mm} \bar{\vx}^k\| \hspace{-0.5mm}+\hspace{-0.5mm} \alpha \|\vy^k \hspace{-0.5mm}-\hspace{-0.5mm} \bar{\vy}^k\| \hspace{-0.5mm}+\hspace{-0.5mm} \alpha L \|\bar{\vx}^k \hspace{-0.5mm}-\hspace{-0.5mm} \tilde{\vx}^{k*}\|,  \label{23b23b39}
\end{align}
where $(a)$ holds because $\rho(\vW - I_{dn})\le 2$, $\rho(\vW) \le 1$ and $\bar{\vy}^k = \vR \nabla F^k(\vx^k)$ as shown in \eqref{tttaaa}, while $(b)$ holds because
$
\|\vR \nabla F^k(\vx^k) - \vR \nabla F^k(\tilde{\vx}^{k*})\| \le  \|\nabla F^k(\vx^k) - \nabla F^k(\tilde{\vx}^{k*})\| \le L \|\vx^k - \tilde{\vx}^{k*}\|.
$
By substituting \eqref{23b23b39} into \eqref{6sd8n} we reach
\begin{align}\label{bound-y1}
&\ \|\vy^{k+1} - \bar{\vy}^{k+1}\| \nonumber \\
\le&\ (\beta + \alpha L)\|\vy^k - \bar{\vy}^k\| + (4L + \alpha L^2)\|\vx^k - \bar{\vx}^k\| \nonumber \\
&\quad + (2L + \alpha L^2)\|\bar{\vx}^k - \tilde{\vx}^{k*}\|^2 +  L\sqrt{n}\Delta_x +  \sqrt{n}  \Delta_g.
\end{align}
If $\alpha \le \frac{1-\beta}{2L}$, it holds that $\beta + \alpha L \le \frac{1+\beta}{2}$ and $\alpha L^2 \le L$. This fact together with \eqref{bound-y1} leads to the result in \eqref{bound-y}.

\section{Proof of Lemma \ref{lm-bound-x-xbar}}\label{app-6}
Taking the average for both sides of recursion \eqref{diging-1}, we have
\begin{align}
\bar{\vx}^{k+1} = \bar{\vx}^k - \alpha \bar{\vy}^k = \vR (\vx^k - \alpha \vy^k). \label{2bs7}
\end{align}
By subtracting \eqref{2bs7} from \eqref{diging-1} we have
\begin{align*}
{\vx}^{k+1} - \bar{\vx}^{k+1} & = (\vW - \vR)\left(\vx^k - \bar{\vx}^k - \alpha(\vy^k - \bar{\vy}^k)\right),
\end{align*}
such that
\begin{align}
& \|{\vx}^{k+1} - \bar{\vx}^{k+1}\| \nonumber \\
& \leq \|(\vW - \vR)(\vx^k - \bar{\vx}^k )\| + \alpha \| (\vW - \vR) ( \vy^k - \bar{\vy}^k)\| \nonumber \\
& \leq \beta \| \vx^k - \bar{\vx}^k \| + \alpha \beta \| \vy^k - \bar{\vy}^k \|.
\end{align}
This is the upper bound in \eqref{bound-x-xbar}.

\section{Proof of Lemma \ref{lm-xbar-xstar}}\label{app-7}
By subtracting $\tilde{\vx}^{(k+1)*}$ from recursion \eqref{2bs7}, we get
\begin{align}\label{sb2387}
&\ \bar{\vx}^{k+1} - \tilde{\vx}^{(k+1)*} \nonumber \\
=&\ \bar{\vx}^k - \tilde{\vx}^{(k+1)*} - \alpha  \bar{\vy}^k \nonumber  \\
\overset{(a)}{=}&\ \bar{\vx}^k - \tilde{\vx}^{k*} + \tilde{\vx}^{k*} - \tilde{\vx}^{(k+1)*} - \alpha \vR \nabla F^k(\vx^k) \nonumber \\
=&\ \bar{\vx}^k - \tilde{\vx}^{k*} - \alpha \vR \nabla F^k(\bar{\vx}^k)  \nonumber \\
&\quad + \alpha \vR(\nabla F^k(\bar{\vx}^k) - \nabla F^k({\vx}^k)) + \tilde{\vx}^{k*} - \tilde{\vx}^{(k+1)*},
\end{align}
where $(a)$ holds because $\bar{\vy}^k = \vR\nabla F^k(\vx^k)$ as shown in \eqref{tttaaa}. Note that when $\alpha \le \frac{2}{\mu + L}$, we have
\begin{align}\label{23bs9826}
&\|\bar{\vx}^k - \tilde{\vx}^{k*} - \alpha \vR \nabla F^k(\bar{\vx}^k)\|^2 \\
=& n \|\bar{x}^k - \tilde{x}^{k*} - \frac{\alpha}{n}\sum_{i=1}^n\nabla f_i^k(\bar{x}^k)\|^2 \nonumber \\
\overset{(a)}{\le}&\ n \left( 1 - \frac{\alpha \mu}{2} \right)^2 \|\bar{x}^k - \tilde{x}^{k*}\|^2 = \left( 1 - \frac{\alpha \mu}{2} \right)^2 \|\bar{\vx}^k - \tilde{\vx}^{k*}\|^2, \nonumber
\end{align}
where $(a)$ follows the derivation of \eqref{xbwe78}. With \eqref{sb2387}, \eqref{23bs9826} and Assumption \ref{ass-smooth}, using the triangle inequality we reach \eqref{lm6-eq}.

\section{Proof of Lemma \ref{lm-7}}\label{app-8}
The argument to prove $\rho(A) \le 1 - \frac{\alpha \mu}{4}$ is adapted from the proof of Lemma 2 in \cite{qu2018harnessing}. The characteristic polynomial $p(\tau)$ of $A$ is derived as
\begin{align}\label{238sdhsdhds7}
p(\tau) \hspace{-1mm}=\hspace{-1mm} [\big( \tau \hspace{-0.5mm}-\hspace{-0.5mm} \frac{1+\beta}{2} \big) \big(\tau \hspace{-0.5mm}-\hspace{-0.5mm} \beta\big) \hspace{-0.5mm}-\hspace{-0.5mm} 5\alpha \beta L] \big( \tau \hspace{-0.5mm}-\hspace{-0.5mm} (1\hspace{-0.5mm}-\hspace{-0.5mm}\frac{\alpha \mu}{2}) \big) \hspace{-0.5mm}-\hspace{-0.5mm} 3\alpha^2 \beta L^2.
\end{align}

Now we define $p_0(\tau) = \big( \tau - \frac{1+\beta}{2} \big) \big(\tau - \beta\big) - 5\alpha \beta L$. The roots for $p_0(\tau)=0$ are given as
\begin{align}
\tau &= \frac{\frac{1+3\beta}{2} \pm \sqrt{(\frac{1+\beta}{2} + \beta)^2 - 4(\frac{1+\beta}{2})\beta + 20\alpha \beta L}}{2} \nonumber \\
&= \frac{\frac{1+3\beta}{2} \pm \sqrt{(\frac{1+\beta}{2} - \beta)^2 + 20\alpha \beta L}}{2}.
\end{align}
We let $\tau_1 < \tau_2$ be two roots of $p_0(\tau)=0$. Apparently, $\tau_1$ and $\tau_2$ are real numbers and it holds that
\begin{align}\label{23nds976}
\tau_1 < \tau_2 \le \frac{\frac{1+3\beta}{2} + \sqrt{(\frac{1+\beta}{2} - \beta)^2 + 20\alpha \beta L}}{2} \le \frac{\beta+3}{4},
\end{align}
when $\alpha \le \frac{3(1-\beta)^2}{80L}$. Relation \eqref{23nds976} implies that
\begin{align}\label{23bsd87}
p_0(\tau) \hspace{-0.8mm}=\hspace{-0.8mm} (\tau \hspace{-0.8mm}-\hspace{-0.8mm} \tau_1)(\tau \hspace{-0.8mm}-\hspace{-0.8mm} \tau_2) \ge (\tau \hspace{-0.8mm}-\hspace{-0.8mm} (\frac{\beta+3}{4}))^2 ~ \mbox{ when } ~ \tau \ge \frac{\beta+3}{4}.
\end{align}
Let $\tau^* = 1 - \frac{\alpha \mu}{4}$. When $\alpha \le \frac{1-\beta}{2\mu}$, it follows that
\begin{align}\label{xbsd623}
\tau^* \ge \max\{ 1 - \frac{\alpha \mu}{4}, \frac{\beta+7}{8} \} \ge \frac{\beta+7}{8} > \frac{\beta+3}{4}.
\end{align}
By substituting $\tau^*$ into \eqref{23bsd87}, we have
\begin{align}\label{23bsd98}
p_0(\tau^*) &\ge (\tau^* \hspace{-0.8mm}-\hspace{-0.8mm} (\frac{\beta+3}{4}))^2 \nonumber \\
&\overset{\eqref{xbsd623}}{\ge} (\frac{\beta+7}{8} - \frac{\beta+3}{4})^2 = \frac{(1-\beta)^2}{64}.
\end{align}
Furthermore, by substituting \eqref{xbsd623} and \eqref{23bsd98} into \eqref{238sdhsdhds7} we reach
\begin{align}
p(\tau^*) {\ge} \frac{\alpha \mu}{4} \frac{(1-\beta)^2}{64} - 3\alpha^2 \beta L^2 \ge 0,
\end{align}
when $\alpha \le \frac{(1-\beta)^2\mu}{768L^2}$. Since $\tau^* > 1 - \frac{\alpha \mu}{2} > \frac{\beta+3}{4}$ (see \eqref{xbsd623}), we have $p_0(\tau) > p_0(\tau^\star)$ when $\tau > \tau^\star$, which implies that
\begin{align}
p(\tau) > p(\tau^*) \ge 0 ~ \mbox{ when } ~ \tau > \tau^*.
\end{align}
It means that there are no real eigenvalues in the interval $(\tau^*, +\infty)$. Since the matrix $A$ is non-negative, it is known from Theorem 8.3.1 of \cite{horn2012matrix} that $\rho(A)$ is a real and non-negative eigenvalue of $A$. As a result, we have
\begin{align}
\rho(A) \le \tau^* = 1 - \frac{\alpha \mu}{4} < 1.
\end{align}
Note that to guarantee the above inequality, we need to make $\alpha$ small enough such that
\begin{align}\label{23bsdbds8}
\alpha \le \min\left\{\frac{3(1-\beta)^2}{80L}, \frac{1-\beta}{2\mu}, \frac{(1-\beta)^2\mu}{768L^2}\right\} = \frac{(1-\beta)^2\mu}{c L^2},
\end{align}
where $c=768$. Since $A$ is stable, the inverse matrix $(I-A)^{-1}$ exists. Note that
\begin{align}\label{xbsbswh}
&(I-A)^{-1} \\
=&\ \ba{ccc}
\frac{1-\beta}{2} & -5 L & -3 L \\
-\alpha \beta & 1-\beta & 0\\
0 & -\alpha L & \frac{\alpha \mu}{2}
\ea^{-1} \nonumber \\
=&\ \frac{1}{C_1}
\ba{ccc}
\frac{\alpha \mu (1-\beta)}{2} & \frac{5\alpha \mu L}{2} + 3\alpha L^2 & 3 L(1-\beta)\\
\frac{\alpha^2 \beta \mu}{2} & \frac{\alpha \mu (1-\beta)}{4} & 3\alpha \beta L \\
\alpha^2 \beta L & \frac{\alpha L(1-\beta)}{2} & \frac{(1-\beta)^2}{2} - 5\alpha \beta L
\ea \nonumber \\
\overset{(a)}{\le}&\ \frac{8}{(1-\beta)^2\alpha \mu}
\ba{ccc}
\frac{\alpha \mu (1-\beta)}{2} & 6\alpha L^2 & 3 L(1-\beta)\\
\frac{\alpha^2 \beta \mu}{2} & \frac{\alpha \mu (1-\beta)}{4} & 3\alpha \beta L \\
\alpha^2 \beta L & \frac{\alpha L(1-\beta)}{2} & \frac{(1-\beta)^2}{2}
\ea, \nonumber
\end{align}
where $C_1 = \frac{(1-\beta)^2\alpha \mu}{4} - 3\alpha^2 \beta L^2 - \frac{5\alpha^2 \beta \mu L}{2}$, inequality $(a)$ holds because $C_1 \ge \frac{(1-\beta)^2\alpha \mu}{4} - 3\alpha^2 L^2 - \frac{5\alpha^2 \mu L}{2} \ge \frac{(1-\beta)^2\alpha \mu}{4} - 6\alpha^2 L^2 \ge \frac{(1-\beta)^2\alpha \mu}{8}$ when $\alpha \le \frac{(1-\beta)^2\mu}{48L^2}$. Apparently, the step-size satisfying \eqref{23bsdbds8} can guarantee this condition.

\section{Proof of Theorem \ref{th-2}}\label{app-9}
With \eqref{23nsd8} and \eqref{xbsbswh}, we have
\begin{align}\label{x23765}
&\ \limsup_{k\to \infty} \ba{c}
\|{\vx}^{k} - \bar{\vx}^{k}\| \\
\|\bar{\vx}^{k} - \tilde{\vx}^{k*}\|
\ea \\
\le&\ C_2
\ba{cc}
\frac{\alpha^2 \beta \mu}{2} & 3\alpha \beta L \\
\alpha^2 \beta L & \frac{(1-\beta)^2}{2}
\ea
\ba{c}
L\sqrt{n}\Delta_x + \sqrt{n} \Delta_g \\
\sqrt{n}\Delta_x
\ea \nonumber \\
=&\ C_2
\ba{c}
\frac{\alpha^2 \beta \mu L\sqrt{n}}{2}\Delta_x + \frac{\alpha^2 \beta \mu \sqrt{n}}{2}\Delta_g + 3\alpha \beta L \sqrt{n} \Delta_x \\
\alpha^2 \beta L^2 \sqrt{n} \Delta_x + \alpha^2 \beta L \sqrt{n} \Delta_g + \frac{(1-\beta)^2\sqrt{n}}{2}\Delta_x
\ea \nonumber \\
\overset{(a)}{\le}&\ C_2
\ba{c}
4\alpha \beta L \sqrt{n} \Delta_x + \alpha^2 \beta L \sqrt{n} \Delta_g  \\
\alpha \beta L \sqrt{n} \Delta_x + \alpha^2 \beta L \sqrt{n} \Delta_g + \frac{(1-\beta)^2\sqrt{n}}{2}\Delta_x
\ea, \nonumber
\end{align}
where $C_2 = \frac{8}{(1-\beta)^2\alpha \mu}$, $(a)$ holds because $\mu \le L$, $\alpha^2\mu/2 \le \alpha $ and $\alpha L \le 1$ when $\alpha \le \frac{(1-\beta)^2 \mu}{768 L^2}$. Inequality \eqref{x23765} implies that
\begin{align}
&\ \limsup_{k\to \infty} \frac{1}{\sqrt{n}} \left( \|{\vx}^{k} - \bar{\vx}^{k}\| + \|\bar{\vx}^{k} - \tilde{\vx}^{k*}\| \right) \nonumber \\
\le&\ \left( \frac{40\beta L}{(1-\beta)^2 \mu} + \frac{4}{\alpha \mu} \right) \Delta_x + \frac{16\alpha \beta L}{(1-\beta)^2\mu} \Delta_g.
\end{align}
This together with $\|{\vx}^{k} - \tilde{\vx}^{k*}\|\le \|{\vx}^{k} - \bar{\vx}^{k}\| + \|\bar{\vx}^{k} - \tilde{\vx}^{k*}\|$ leads to \eqref{eq-thm-diging}.

\bibliographystyle{IEEEbib}
\bibliography{reference}

\begin{thebibliography}{10}

\bibitem{jakubiec2013d}
F.~Y. Jakubiec and A.~Ribeiro,
\newblock ``D-map: Distributed maximum a posteriori probability estimation of
  dynamic systems,''
\newblock {\em IEEE Transactions on Signal Processing}, vol. 61, no. 2, pp.
  450--466, 2013.

\bibitem{tu2014distributed}
S.-T. Tu and A.~H. Sayed,
\newblock ``Distributed decision-making over adaptive networks,''
\newblock {\em IEEE Transactions on Signal Processing}, vol. 62, no. 5, pp.
  1054--1069, 2014.

\bibitem{Rahili2017dynamic}
S.~Rahili and W.~Ren,
\newblock ``Distributed continuous-time convex optimization with time-varying
  cost functions,''
\newblock {\em IEEE Transactions on Automatic Control}, vol. 62, no. 4, pp.
  1590--1605, 2018.

\bibitem{maros2018beamforming}
M.~Maros and J.~Jalden,
\newblock ``Admm for distributed dynamic beam-forming,''
\newblock {\em IEEE Transactions on Signal and Information Processing over
  Networks}, vol. 4, no. 2, pp. 220--235, 2018.

\bibitem{tang2019time}
Y.~Tang,
\newblock {\em Time-Varying Optimization and Its Application to Power System
  Operation},
\newblock Ph.D. thesis, California Institute of Technology, 2019.

\bibitem{emiliano2019survey}
E.~Dall'Anese, A.~Simonetto, S.~Becker, and L.~Madden,
\newblock ``Optimization and learning with information streams: Time-varying
  algorithms and applications,''
\newblock {\em arXiv preprint arXiv:1910.08123}, 2019.

\bibitem{sayed2014adaptive}
A.~H. Sayed,
\newblock ``Adaptive networks,''
\newblock {\em Proceedings of the IEEE}, vol. 102, no. 4, pp. 460--497, 2014.

\bibitem{sayed2014adaptation}
A.~H. Sayed,
\newblock ``Adaptation, learning, and optimization over networks,''
\newblock {\em Foundations and Trends in Machine Learning}, vol. 7, no. 4--5,
  pp. 311--801, 2014.

\bibitem{nedic2009distributed}
A.~Nedi{\'c} and A.~Ozdaglar,
\newblock ``Distributed subgradient methods for multi-agent optimization,''
\newblock {\em IEEE Transactions on Automatic Control}, vol. 54, no. 1, pp.
  48--61, 2009.

\bibitem{yuan2016convergence}
K.~Yuan, Q.~Ling, and W.~Yin,
\newblock ``On the convergence of decentralized gradient descent,''
\newblock {\em SIAM Journal on Optimization}, vol. 26, no. 3, pp. 1835--1854,
  2016.

\bibitem{duchi2012dual}
J.~C. Duchi, A.~Agarwal, and M.~J. Wainwright,
\newblock ``Dual averaging for distributed optimization: convergence analysis
  and network scaling,''
\newblock {\em IEEE Transactions on Automatic Control}, vol. 57, no. 3, pp.
  592--606, 2012.

\bibitem{mokhtari2017network}
A.~Mokhtari, Q.~Ling, and A.~Ribeiro,
\newblock ``Network {N}ewton distributed optimization methods,''
\newblock {\em IEEE Transactions on Signal Processing}, vol. 65, no. 1, pp.
  146--161, 2017.

\bibitem{bajovic2017second}
D.~Bajovic, D.~Jakovetic, N.~Krejic, and N.~K. Jerinkic,
\newblock ``Newton-like method with diagonal correction for distributed
  optimization,''
\newblock {\em SIAM Journal on Optimization}, vol. 27, no. 2, pp. 1171--1203,
  2017.

\bibitem{mateos2010distributed}
G.~Mateos, J.~A. Bazerque, and G.~B. Giannakis,
\newblock ``Distributed sparse linear regression,''
\newblock {\em IEEE Transactions on Signal Processing}, vol. 58, no. 10, pp.
  5262--5276, 2010.

\bibitem{mota2013d}
J.~F. Mota, J.~M. Xavier, P.~M. Aguiar, and M.~P{\"u}schel,
\newblock ``{D-ADMM}: A communication-efficient distributed algorithm for
  separable optimization,''
\newblock {\em IEEE Transactions on Signal Processing}, vol. 61, no. 10, pp.
  2718--2723, 2013.

\bibitem{shi2014linear}
W.~Shi, Q.~Ling, K.~Yuan, G.~Wu, and W.~Yin,
\newblock ``On the linear convergence of the {ADMM} in decentralized consensus
  optimization,''
\newblock {\em IEEE Transactions on Signal Processing}, vol. 62, no. 7, pp.
  1750--1761, 2014.

\bibitem{ling2015dlm}
Q.~Ling, W.~Shi, G.~Wu, and A.~Ribeiro,
\newblock ``{DLM}: Decentralized linearized alternating direction method of
  multipliers,''
\newblock {\em IEEE Transactions on Signal Processing}, vol. 63, no. 15, pp.
  4051--4064, 2015.

\bibitem{chang2015multi}
T.-H. Chang, M.~Hong, and X.~Wang,
\newblock ``Multi-agent distributed optimization via inexact consensus
  {ADMM},''
\newblock {\em IEEE Transactions on Signal Processing}, vol. 63, no. 2, pp.
  482--497, 2015.

\bibitem{shi2015extra}
W.~Shi, Q.~Ling, G.~Wu, and W.~Yin,
\newblock ``{EXTRA}: An exact first-order algorithm for decentralized consensus
  optimization,''
\newblock {\em SIAM Journal on Optimization}, vol. 25, no. 2, pp. 944--966,
  2015.

\bibitem{yuan2018exact}
K.~Yuan, B.~Ying, X.~Zhao, and A.~H. Sayed,
\newblock ``Exact diffusion for distributed optimization and learning---{P}art
  i: Algorithm development,''
\newblock {\em IEEE Transactions on Signal Processing}, vol. 67, no. 3, pp.
  708--723, 2018.

\bibitem{li2017decentralized}
Z.~Li, W.~Shi, and M.~Yan,
\newblock ``A decentralized proximal-gradient method with network independent
  step-sizes and separated convergence rates,''
\newblock {\em IEEE Transactions on Signal Processing}, vol. 67, no. 17, pp.
  4494--4506, 2019.

\bibitem{xu2015augmented}
J.~Xu, S.~Zhu, Y.~C. Soh, and L.~Xie,
\newblock ``Augmented distributed gradient methods for multi-agent optimization
  under uncoordinated constant stepsizes,''
\newblock in {\em IEEE Conference on Decision and Control (CDC)}, Osaka, Japan,
  2015, pp. 2055--2060.

\bibitem{di2016next}
P.~Di~Lorenzo and G.~Scutari,
\newblock ``Next: In-network nonconvex optimization,''
\newblock {\em IEEE Transactions on Signal and Information Processing over
  Networks}, vol. 2, no. 2, pp. 120--136, 2016.

\bibitem{nedic2017achieving}
A.~Nedic, A.~Olshevsky, and W.~Shi,
\newblock ``Achieving geometric convergence for distributed optimization over
  time-varying graphs,''
\newblock {\em SIAM Journal on Optimization}, vol. 27, no. 4, pp. 2597--2633,
  2017.

\bibitem{qu2018harnessing}
G.~Qu and N.~Li,
\newblock ``Harnessing smoothness to accelerate distributed optimization,''
\newblock {\em IEEE Transactions on Control of Network Systems}, vol. 5, no. 3,
  pp. 1245--1260, 2018.

\bibitem{xin2018linear}
R.~Xin and U.~A. Khan,
\newblock ``A linear algorithm for optimization over directed graphs with
  geometric convergence,''
\newblock {\em IEEE Control Systems Letters}, vol. 2, no. 3, pp. 315--320,
  2018.

\bibitem{towfic2015stability}
Z.~J. Towfic and A.~H. Sayed,
\newblock ``Stability and performance limits of adaptive primal-dual
  networks,''
\newblock {\em IEEE Transactions on Signal Processing}, vol. 63, no. 11, pp.
  2888--2903, 2015.

\bibitem{tang2018d}
H.~Tang, X.~Lian, M.~Yan, C.~Zhang, and J.~Liu,
\newblock ``D2: Decentralized training over decentralized data,''
\newblock in {\em International Conference on Machine Learning (ICML)},
  Stockholm, Sweden, 2018, pp. 1--8.

\bibitem{yuan2019performance}
K.~Yuan, S.~A. Alghunaim, B.~Ying, and A.~H. Sayed,
\newblock ``On the influence of bias-correction on distributed stochastic
  optimization,''
\newblock {\em arXiv preprint arXiv:1903.10956}, 2019.

\bibitem{pu2018distributed}
S.~Pu and A.~Nedi{\'c},
\newblock ``A distributed stochastic gradient tracking method,''
\newblock in {\em IEEE Conference on Decision and Control (CDC)}, Miami, USA,
  2018, pp. 963--968.

\bibitem{Xin2019introduction}
R.~Xin, S.~Kar, and U.~A. Khan,
\newblock ``An introduction to decentralized stochastic optimization with
  gradient tracking,''
\newblock {\em arXiv preprint arXiv:1907.09648}, 2019.

\bibitem{sun2017distributed}
C.~Sun, M.~Ye, and G.~Hu,
\newblock ``Distributed time-varying quadratic optimization for multiple agents
  under undirected graphs,''
\newblock {\em IEEE Transactions on Automatic Control}, vol. 62, no. 7, pp.
  3687--3694, 2017.

\bibitem{xi2016distributed}
C.~Xi and U.~A. Khan,
\newblock ``Distributed dynamic optimization over directed graphs,''
\newblock in {\em IEEE Conference on Decision and Control (CDC)}, Las Vegas,
  USA, 2016, pp. 245--250.

\bibitem{liu2017decentralized}
H.~J. Liu, W.~Shi, and H.~Zhu,
\newblock ``Decentralized dynamic optimization for power network voltage
  control,''
\newblock {\em IEEE Transactions on Signal and Information Processing over
  Networks}, vol. 3, no. 3, pp. 568--579, 2017.

\bibitem{simonetto2017decentralized}
A.~Simonetto, A.~Koppel, A.~Mokhtari, G.~Leus, and A.~Ribeiro,
\newblock ``Decentralized prediction-correction methods for networked
  time-varying convex optimization,''
\newblock {\em IEEE Transactions on Automatic Control}, vol. 62, no. 11, pp.
  5724--5738, 2017.

\bibitem{ling2013decentralized}
Q.~Ling and A.~Ribeiro,
\newblock ``Decentralized dynamic optimization through the alternating
  direction method of multipliers,''
\newblock {\em IEEE Transactions on Signal Processing}, vol. 62, no. 5, pp.
  1185--1197, 2013.

\bibitem{mokhtari2016decentralized}
A.~Mokhtari, W.~Shi, Q.~Ling, and A.~Ribeiro,
\newblock ``A decentralized second-order method for dynamic optimization,''
\newblock in {\em IEEE Conference on Decision and Control (CDC)}, Las Vegas,
  USA, 2016, pp. 6036--6043.

\bibitem{simonetto2014double}
A.~Simonetto and G.~Leus,
\newblock ``Double smoothing for time-varying distributed multiuser
  optimization,''
\newblock in {\em IEEE Global Conference on Signal and Information Processing
  (GlobalSIP)}, Atlanta, USA, 2014, pp. 852--856.

\bibitem{chen2013distributed}
J.~Chen and A.~H. Sayed,
\newblock ``Distributed {Pareto} optimization via diffusion strategies,''
\newblock {\em IEEE Journal of Selected Topics in Signal Processing}, vol. 7,
  no. 2, pp. 205--220, 2013.

\bibitem{towfic2013distributed}
Z.~J. Towfic, J.~Chen, and A.~H. Sayed,
\newblock ``On distributed online classification in the midst of concept
  drifts,''
\newblock {\em Neurocomputing}, vol. 112, pp. 138--152, 2013.

\bibitem{berahas2018balancing}
A.~Berahas, R.~Bollapragada, N.~S. Keskar, and E.~Wei,
\newblock ``Balancing communication and computation in distributed
  optimization,''
\newblock {\em IEEE Transactions on Automatic Control}, vol. 64, no. 8, pp.
  3141--3155, 2018.

\bibitem{li2018sharp}
H.~Li, C.~Fang, W.~Yin, and Z.~Lin,
\newblock ``A sharp convergence rate analysis for distributed accelerated
  gradient methods,''
\newblock {\em arXiv preprint arXiv:1810.01053}, 2018.

\bibitem{zhu2010discrete}
M.~Zhu and S.~Martinez,
\newblock ``Discrete-time dynamic average consensus,''
\newblock {\em Automatica}, vol. 46, no. 2, pp. 322--329, 2010.

\bibitem{simonetto2016class}
A.~Simonetto, A.~Mokhtari, A.~Koppel, G.~Leus, and A.~Ribeiro,
\newblock ``A class of prediction-correction methods for time-varying convex
  optimization,''
\newblock {\em IEEE Transactions on Signal Processing}, vol. 64, no. 17, pp.
  4576--4591, 2016.

\bibitem{popkov2005gradient}
A.~Y. Popkov,
\newblock ``Gradient methods for nonstationary unconstrained optimization
  problems,''
\newblock {\em Automation and Remote Control}, vol. 66, no. 6, pp. 883--891,
  2005.

\bibitem{yuan2017exact1}
K.~Yuan, B.~Ying, X.~Zhao, and A.~H. Sayed,
\newblock ``Exact dffusion for distributed optimization and learning -- {Part
  I: Algorithm development},''
\newblock {\em IEEE Transactions on Signal Processing}, vol. 67, no. 3, pp. 708
  -- 723, 2019.

\bibitem{nesterov2004introductory}
Y.~Nesterov,
\newblock {\em Introductory Lectures on Convex Optimization: A Basic Course},
\newblock Springer Science \& Business Media, 2004.

\bibitem{horn2012matrix}
R.~A. Horn and C.~R. Johnson,
\newblock {\em Matrix Analysis},
\newblock Cambridge University Press, 2012.

\end{thebibliography}

\end{document}